\documentclass[11pt]{article}
\usepackage[utf8]{inputenc} 
\usepackage[T1]{fontenc}    
\usepackage{hyperref}       
\usepackage{url}            
\usepackage{microtype}
\usepackage{graphicx}
\usepackage{subcaption}
\usepackage{array}
\usepackage{color} 
\usepackage{tabularx}
\usepackage{amsmath,mathdots,amsthm}
\usepackage{amssymb}
\usepackage{amsfonts}
\usepackage{txfonts}
\usepackage{multicol}
\usepackage{arxiv}
\usepackage[dvipsnames]{xcolor}
\usepackage{bm}
\usepackage{soul}
\usepackage{float}
\usepackage{cite}
\usepackage[colorinlistoftodos]{todonotes}
\usepackage[dvipsnames]{xcolor}
\usepackage{enumitem} 
\usepackage{tikz}
\usetikzlibrary{calc} 

\hypersetup{
	colorlinks=true,                          
	linkcolor=blue, 
	citecolor=blue, 
	urlcolor=black  
}

\newtheorem{theorem}{\bfseries Theorem}[section]

\newtheorem{rmrk}{Remark}
\newtheorem*{rmrk*}{Remark}

\newtheorem*{definition*}{Definition}

\newtheorem*{thm*}{Theorem}

\newtheorem{lmm}{Lemma}
\newtheorem*{lmm*}{Lemma}

\newtheorem{prp}{Proposition}
\newtheorem*{prp*}{Proposition}

\newtheorem{prop}{Property}
\newtheorem*{prop*}{Property}

\DeclareMathAlphabet{\mathcal}{OMS}{cmsy}{m}{n}
\SetMathAlphabet{\mathcal}{bold}{OMS}{cmsy}{b}{n}
\newcommand{\bigO}{\mathcal{O}}

\newcommand{\nrm}[1]{ \left\vert\left\vert #1 \right\vert\right\vert }

\newcommand{\nrmLp}[1]{\left\| #1 \right\|_{L^2}}
\newcommand{\pd}[2]{\frac{\partial #1}{\partial #2}}
\newcommand{\prn}[1]{\left( #1 \right)}
\newcommand{\pdt}[1]{\frac{\partial #1}{\partial t}}

\newcommand{\half}{\frac{1}{2}}
\newcommand{\curl}{\boldsymbol{\omega}}
\renewcommand{\div}{\nabla\cdot}
\newcommand{\T}{^\mathsf{T}}

\renewcommand{\d}{\,\mathrm{d}}

\renewcommand{\L}{\mathcal{L}}
\newcommand{\I}{\mathbf{I}}

\newcommand{\x}{\mathbf{x}}

\newcommand{\en}{\mathcal{E}}
\newcommand{\toten}{\mathrm{E}}

\newcommand{\V}{\mathcal{V}}
\newcommand{\rot}{\nabla\times}
\newcommand{\bn}{\mathbf{n}} 
\newcommand{\bx}{\mathbf{x}} 
\newcommand{\nablacp}{\nabla_c^p}
\newcommand{\nablacq}{\nabla_c^q}
\newcommand{\nablapc}{\nabla_p^c}
\newcommand{\rotcp}{\nablacp\times}
\newcommand{\rotpc}{\nablapc\times}
\newcommand{\divcp}{\nablacp\cdot}
\newcommand{\divpc}{\nablapc\cdot}
\newcommand{\J}{\mathbf{J}} %
\newcommand{\Js}{\J^{\star}} %
\newcommand{\ppsi}{\boldsymbol{\psi}}
\newcommand{\ppsis}{\ppsi^{\star}}
\newcommand{\A}{\mathbf{A}} 
\newcommand{\B}{\mathbf{B}} 
\newcommand{\Q}{\mathbf{Q}} 
\newcommand{\RR}{\mathbb{R}}

\newcommand{\wpJ}{\curl_p^{\J}}
\newcommand{\wcpsi}{\curl_c^{\ppsi}}

\newcommand{\wpJn}{\curl_p^{\J^n}}
\newcommand{\wpJnp}{\curl_p^{\J^{n+1}}}
\newcommand{\wpJnph}{\curl_p^{\J^{n+\frac12}}}
\newcommand{\wcpsin}{\curl_c^{\ppsi^n}}
\newcommand{\wcpsinp}{\curl_c^{\ppsi^{n+1}}}
\newcommand{\wcpsinph}{\curl_c^{\ppsi^{n+\frac12}}}

\newcommand{\Ediss}{\mathrm{D}^{n+\frac12}}
\newcommand{\dissipterm}{\displaystyle\dfrac{\Delta t}{\tau}\sum_{c\in\Omega} |\Omega_c|\nrm{\J_c^{n+\half}}^2}
\newcommand{\Edisscurl}{\mathrm{D}_{\curl}^{n+\frac12}}
\newcommand{\dissiptermcurl}{\displaystyle\frac{\Delta t}{\tau} \sum_{p\in\Omega}|\Omega_p| \nrm{\wpJnph}^2}
\newcommand{\Edissdiv}{\mathrm{D}_{\nabla}^{n+\frac12}}
\newcommand{\dissiptermdiv}{\displaystyle\frac{\Delta t}{\tau} \sum_{p\in\Omega} |\Omega_p|\nrm{\divpc \J_c^{n+\half}}^2}
\newcommand{\halb}{\frac{1}{2}}

\graphicspath{ {./figures/} }

\title{Analysis and structure-preserving discretization of the hyperbolic heat-GLM system}

\author{
	Firas Dhaouadi$^a$, Laura R\'io-Mart\'in$^{b}$, Michael Dumbser$^{c,d}$ 
}

\affiliation{
	$^a$ Bordeaux INP, Institut de Math\'ematiques de Bordeaux, Universit\'e de Bordeaux, CNRS UMR 5251, INRIA, 33405, Talence, France\\
	$^b$ INRIA, Universit\'e de Bordeaux, Institut de Math\'ematiques de Bordeaux, CNRS UMR 5251, 33405, Talence, France \\
	$^c$ Laboratory of Applied Mathematics, University of Trento, Via Mesiano 77, 38123 Trento, Italy \\
	$^d$ Department of Mathematics, Southern University of Science and Technology, \\ Xueyuan Avenue 1088, 518055 Shenzhen, Guangdong, China
}
\date{}

\renewcommand{\headeright}{Dhaouadi, R\'io-Mart\'in, Dumbser}
\renewcommand{\undertitle}{}
\renewcommand{\shorttitle}{}

\hypersetup{
	pdftitle={Analysis and structure-preserving discretization of the hyperbolic heat-GLM system},
	pdfsubject={Structure-preserving numerical methods for hyperbolic PDEs},
	pdfauthor={Firas Dhaouadi, Laura Rio-Martin, Michael Dumbser},
	pdfkeywords={Hyperbolic equations, structure-preserving schemes, asymptotic-preserving schemes, energy consistency},
}

\begin{document}
	\maketitle

	\begin{abstract}
		In this paper, we study a hyperbolic prototype system arising from the coupling of linear acoustics with Generalized-Lagrangian-Multiplier (GLM) curl-cleaning and a relaxation source term. 
		We first show that, in the absence of source terms, the equations can be derived from a variational principle and that the resulting system is symmetric-hyperbolic, conserves the total energy and admits an interesting set of invariants. When present, the relaxation source term describes Cattaneo-type heat transfer processes that are asymptotically compatible with the Fourier law in the stiff relaxation limit. In this case, the total energy and the additional invariants are dissipated accordingly.
		A new semi-implicit compatible numerical scheme on staggered grids is developed to solve this system while exactly preserving its key properties at the discrete level. 
		In particular, we prove that the scheme is asymptotic-preserving (AP) in the Fourier limit, with a convergence rate depending on the well-preparedness of the initial data. 
		We also demonstrate that the scheme conserves exactly the total energy as well as all the invariants in the absence of source terms, and dissipates them consistently in the presence of relaxation. 
		The theoretical findings are illustrated on a set of representative numerical test cases.
	\end{abstract}
	\section{Introduction}
	
	Many physically relevant models, such as Maxwell's equations~\cite{maxwell1865viii}, magnetohydrodynamics (MHD)~\cite{alfven1942existence}, and various systems of continuum mechanics and physics, possess involutions, that is, stationary differential constraints on specific fields (typically of the divergence or rotational type) that are automatically preserved by the continuous PDE whenever they hold initially. In numerical simulations, however, there is no guarantee that a discretization will automatically preserve these constraints, which may lead to unphysical or even unstable numerical solutions. This classical issue has spurred the development of specialized numerical techniques that allow the enforcement or control of such constraints, see for example \cite{yee1966numerical,holland1983finite,brecht1981simulation,evans1988simulation,devore1991flux,dai1998simple,toth2000b,gardiner2005unsplit,balsaraSpicer1999,balsara2004,xu2016divergence,hazra2019globally,balsara2020making,balsara2023CurlFree}. 
	
	Among these methodologies, the Generalized-Lagrangian Multiplier (GLM) technique is a distinctive approach that appeared in the pioneering work of Munz~\textit{et al.} on the Maxwell and MHD equations~\cite{munz2000divergence,dedner2002,dedner2003divergence}, as a practical strategy to control numerical divergence errors. The originality of the approach lies in the fact that the continuous PDE system itself is modified in order to tackle a numerical issue. In particular, an artificial field is introduced to advect any spurious divergence errors to the boundaries, thus preventing their accumulation. This idea was later also generalized to hyperbolic curl-cleaning methods~\cite{dumbser2020glm,chiocchetti2021SurfaceTension,busto2021HyperbolicDispersion,riomartin2024ADERDG}.
	
	The effectiveness of GLM-curl cleaning has been documented in the literature. Indeed, independent numerical experiments indicate that it can prevent the growth of nonphysical modes that may otherwise lead to finite-time blow-up of the discrete solution~\cite{chiocchetti2021SurfaceTension,dhaouadi2023SPNSK}. In particular, in \cite{chiocchetti2021SurfaceTension}, simulations are shown to blow up immediately in the absence of curl-cleaning or any discrete curl-preserving techniques. This may be due to the fact that GLM-curl cleaning also has the remarkable advantage of restoring strong hyperbolicity in multiple space dimensions for a wide class of PDEs as shown for example in \cite{chiocchetti2021SurfaceTension,busto2021HyperbolicDispersion,riomartin2024ADERDG,dhaouadi2022NSK}. This occurs by converting the defective characteristic field associated with the continuum velocity into additional propagating wave families governed by the cleaning wave speed, yielding a complete eigensystem in multiple dimensions. 
	
	Nevertheless, existing results mainly rely on numerical experimentation and typically report a decrease in the magnitude of curl errors as the cleaning speed increases. However, there does not seem to be any established convergence rate or theoretical proof among published works~\cite{dhaouadi2022NSK,busto2021HyperbolicDispersion,chiocchetti2021SurfaceTension}.  
		
	In this context, the current contribution rigorously addresses this phenomenon at both the continuous and discrete levels, while also shedding light on an original system of PDEs coupling a classical hyperbolic Cattaneo-type heat conduction system with GLM-curl cleaning, which we shall refer to as the hyperbolic heat-GLM system. We demonstrate that the governing equations, while seemingly simple, exhibit a rich variety of nontrivial structural properties. These include symmetric $t$-hyperbolicity in the sense of Friedrichs, curl-free and divergence-free involutions, asymptotic compatibility with Fourier's law of heat conduction, as well as the existence of non-trivial invariants in addition to the total energy of the system. All these features make the model particularly appealing as a prototype system for benchmarking general structure-preserving (SP) numerical methods for hyperbolic PDEs.   
	
	While structure-preserving numerical methods have been around for more than sixty years, dating back to the work of Yee \cite{yee1966numerical} for the Maxwell equations, they continue to evolve through many contributions over the years, some of which are listed here as examples, without pretending to be exhaustive: constrained-transport methods for divergence-free magnetic fields \cite{balsaraSpicer1999,balsara2004,balsara2020making,balsara2023CurlFree}, exactly curl-preserving schemes on staggered grids \cite{boscheri2021structure,dhaouadi2023SPNSK,riomartin2025curlfree,chiocchetti2023exactly}, also with a stronger accent on thermodynamics \cite{boscheri2024new,boscheri2026structure,boscheridhaouadi2026heat}, general frameworks embedding div-curl-grad constraints into discontinuous Galerkin and general unstructured mesh discretizations  \cite{perrier2025development,abgrall2025simple,abgrall2026embedding}, and other recent contributions \cite{jung2024curl,bernardelli2026structure,barsukow2025structure,ranocha2025structure}.
	
	In the current contribution, we construct a semi-implicit fully discrete scheme that provably preserves, at the discrete level, all structural properties admitted by the considered equations. The scheme generalizes the approach presented in \cite{dumbser2025maxwellGLM}, which is closely related to the general multidimensional methods developed in \cite{abgrall2025simple,abgrall2026embedding}, and extends it to the case where a relaxation source term is also present. 
	The spatial discretization is based on a staggered grid in which the system variables whose time-evolution is governed by an Euler--Lagrange equation are stored at the vertices of the primal mesh, while the rest of the variables are stored in the corresponding cell centers, allowing us to define compatible discrete gradient, divergence, or curl operators. By compatible, we mean that the identities $\nabla\cdot(\nabla\times\cdot)=0$ and $\nabla\times\nabla(\cdot)=0$ also hold exactly at the discrete level. 
	A Crank--Nicolson method is used for the time discretization. In this case, this not only allows for an unconditionally $L^2$-stable second-order scheme, but also conserves all the considered invariants of the system. In the case where the dissipative source term is also present, we show that some of these invariants become Lyapunov functionals at the continuous level and are also dissipated consistently at the discrete level.
	Finally, we prove that the fully discrete scheme is asymptotic-preserving in the Fourier limit, with a convergence rate that depends on the initial data. 
	
	The paper is organized as follows. In Section~\ref{sec:gov_eq}, we introduce the heat-GLM system and analyze its structural properties in both the dissipative and dissipationless cases. In particular, we show how to recover the dissipationless part of the system from an underlying variational principle, using a Lagrangian depending on a set of carefully chosen auxiliary variables. We then demonstrate that the system is symmetric hyperbolic and that adding a relaxation source term results in a dissipative system that asymptotically reduces to a classical parabolic heat equation for a particular scaling of the relaxation time with the wave speed. 
	Section~\ref{sec:cont_constraints} is devoted to the analysis of the conservation of some additional quadratic quantities at the continuous level, including total energy, rotational energy, and some wave-equation invariants. In Section~\ref{sec:num_discretization}, we describe the staggered semi-implicit scheme, define the compatible discrete operators, and prove its structure-preserving properties at the discrete level. Finally, Section~\ref{sec:num_results} reports numerical simulations that illustrate the accuracy of the method, verify the theoretical convergence rates, and confirm the discrete conservation and dissipation properties.
	
	\section{Governing equations}\label{sec:gov_eq}
	\subsection{The (reversible) heat-GLM system}\label{sec:reversible_sys}
	Consider the following system of equations
	\begin{subequations}
		\begin{align}
			& \pdt{T} + c_0\, \nabla \cdot\J = 0, \label{eq:T} \\
			& \pdt{\J} +c_0\nabla T + c_h \rot \ppsi = 0, \label{eq:J} \\
			& \pdt{\ppsi}- c_h\rot \J + c_h \nabla \varphi = 0, \label{eq:psi} \\
			& \pdt{\varphi} + c_h \nabla \cdot \ppsi = 0. \label{eq:phi}
		\end{align}
		\label{eq:heatGLM}
	\end{subequations}
	All the evolved quantities are assumed to be functions of $(\mathbf{x},t) \in \Omega\times[0,t_f]\subset \RR^d \times \mathbb{R}_0^+$, where $d$ is the number of space dimensions and $t_f>0$ is a fixed final time. We shall refer to this system as the hyperbolic heat-GLM system. Its structure is reminiscent of, and indeed formally analogous to, the so-called Maxwell--GLM or Maxwell--Munz equations \cite{munz2000divergence,dumbser2025maxwellGLM}, although they differ in purpose and physical meaning. In fact, the latter was designed to propagate numerical divergence errors produced by non-structure-preserving schemes applied to the evolution equations of the magnetic and electric fields, while the former propagates spurious curl errors of $\J$ and $\ppsi$. Beyond this analogy, the system remains general and several classical models are embedded within the present formulation as special cases. For example, setting $\varphi = T = 0$ with initially divergence-free $\J$ recovers the Maxwell equations, while setting $\ppsi = 0, \varphi = 0$ with initially curl-free $\J$ reduces to linear acoustics, see also \cite{abgrall2025simple}. 
	
	In what follows, we will refer to $T\in\RR$ as the temperature field and to $\J$ as the heat flux (in some scaling), although these designations are meaningful only in the presence of the suitable relaxation source terms that we introduce later. The variables  $\varphi\in\RR$ and $\ppsi\in\RR^d$ are the auxiliary GLM fields. The parameters $c_0\in\RR$ and $c_h\in \RR$ are constant wave speeds.   
	In this setting, the system \eqref{eq:heatGLM} describes the wave propagation of the scalar fields $T,\varphi\in\RR$ and their respective normalized fluxes $\J,\ppsi\in\RR^d$. The coupling between the $(T,\J)$ and $(\varphi,\ppsi)$ subsystems occurs only through the rotational terms in the equations (\ref{eq:J})-(\ref{eq:psi}). In particular, if both $\J$ and $\ppsi$ are initially curl-free, Property~\ref{prop:curl_energy} below shows that they remain so for all times, and the two subsystems then evolve independently. 
	
	\subsection{Derivation from a variational principle}
	System \eqref{eq:heatGLM} can be derived from Hamilton's principle of stationary action, from a Lagrangian expressed in terms of auxiliary potentials. To this end, let us consider $Z(\x,t)\in\RR$ and $\A(\x,t)\in\RR^d$ such that 
	\begin{subequations}
		\begin{align}
			&T = -Z_t\, , \label{def:T} \\
			&\ppsi = -\A_t\, ,   \label{def:psi}
		\end{align}
		\label{def:T_psi}
	\end{subequations}
	and let us build the following Lagrangian density
	\begin{equation}
		\Lambda\prn{\nabla Z, Z_t,\A_t,\nabla\cdot\A,\nabla\times\A} = \half\prn{Z_t^2+ \nrm{\A_t}^2 - \nrm{ c_0\nabla Z + c_h \rot \A}^2  \, - c_h^2\prn{\nabla\cdot \A}^2}, 
		\label{eq:Lagrangian}
	\end{equation}
	which differs from the one used in \cite{dumbser2025maxwellGLM} only by a divergence term and therefore yields the same action. Indeed, we have   
	\begin{equation}
		\nrm{ c_0\nabla Z + c_h \rot \A}^2 = 
		\nrm{c_0\nabla Z}^2 + \nrm{ c_h \rot \A}^2 +  2 c_0 c_h \nabla Z \cdot \rot \A 
	\end{equation}
	Using the standard vector calculus identity 
	\begin{equation*}
		\nabla \cdot \left( \A \times \B \right) = 
		\left( \nabla \times \A \right) \cdot \B - \A \cdot \left( \nabla \times \B \right)
	\end{equation*}
	for two vector fields $\A$ and $\B$  with $\B = \nabla Z$ yields 
		\begin{equation}
		\nabla \cdot \left( \A \times \nabla Z \right) = 
		\left( \nabla \times \A \right) \cdot \nabla Z 
	\end{equation}
	since $\nabla \times \nabla Z = 0$. Hence, the new Lagrangian density \eqref{eq:Lagrangian} is equivalent to the original Lagrangian  
		\begin{equation}
		\Lambda_o\prn{\nabla Z, Z_t,\A_t,\nabla\cdot\A,\nabla\times\A} = \half\prn{Z_t^2+ \nrm{\A_t}^2 - \nrm{ c_0\nabla Z}^2 - \nrm{c_h \rot \A}^2  \, - c_h^2\prn{\nabla\cdot \A}^2}, 
		\label{eq:Lagrangian.org}
	\end{equation}
	introduced in \cite{dumbser2025maxwellGLM}, as the two only differ in the divergence term $2 c_0 c_h \nabla \cdot \left( \A \times \nabla Z \right) $. 
	In this case, it is straightforward to obtain the Euler--Lagrange equations by considering the action 
	\begin{equation*}
		\mathcal{A}[Z,\A] = \int_{t_0}^{t_1} \int_\Omega \Lambda\prn{ Z_t,\nabla Z,\A_t,\nabla\cdot\A,\nabla\times\A} \, \mathrm{d}\Omega, \qquad 0\leq t_0<t_1,
	\end{equation*}
	and applying directly Hamilton's principle to it, which yields for both the variations of $Z$ and $\A$, respectively
	\begin{subequations}
		\begin{align}
			\pdt{}\pd{\Lambda}{Z_t}  + \nabla\cdot \prn{\pd{\Lambda}{\nabla Z}} &= \pdt{Z_t} - c_0\, \nabla\cdot \prn{c_0\nabla Z + c_h \rot \A} = 0, \label{eq:EL-Z} \\
			\pdt{}\pd{\Lambda}{\A_t} - \rot \prn{\pd{\Lambda}{\nabla \times \A}} + \nabla \prn{\pd{\Lambda}{\nabla \cdot \A}} & = \pdt{\A_t}+  c_h \rot \prn{c_0\, \nabla Z + c_h \rot \A } - c_h^2 \, \nabla \prn{\nabla \cdot \A} = 0.
			\label{eq:EL-A}
		\end{align}
	\end{subequations}
	Now, in order to cast these equations into first-order in both space and time, we make use of the definitions \eqref{def:T_psi}, and we additionally define the new fields $\J$ and $\varphi$ as  
	\begin{equation}
	\J = c_0\nabla Z + c_h \rot \A,  \qquad
	\varphi = c_h \nabla\cdot \A. 
	\label{def:J_phi}
	\end{equation}
	Differentiating both sides of these equations with respect to time and using the definitions \eqref{def:T_psi} allows us to obtain
	\begin{equation*}
		\pdt{\J} = -c_0\nabla T - c_h \rot\ppsi, \qquad \pdt\varphi = -c_h \nabla\cdot \ppsi, 
	\end{equation*}
	which correspond exactly to \eqref{eq:J} and \eqref{eq:phi}, respectively. The remaining equations, \textit{i.e.}, \eqref{eq:T} and \eqref{eq:psi}, simply correspond to the Euler--Lagrange equations \eqref{eq:EL-Z} and \eqref{eq:EL-A}, respectively, upon substituting the definitions \eqref{def:T_psi} into them. Therefore, under these notations, one recovers exactly the reversible hyperbolic heat-GLM system~\eqref{eq:heatGLM} as two pairs of Euler--Lagrange equations, each coming with a so-called \textit{trivial consequence equation}, obtained from the order-reduction definitions \eqref{def:J_phi}. 
	\begin{rmrk}
		The definition of the temperature field as the time derivative of an auxiliary scalar field $Z$, usually referred to as the thermal displacement, is recurrent in the literature and dates back to the works of Helmholtz \cite{Helmholtz1884}, having been subsequently developed, for example, in~\cite{taub1949hamilton,herivel1955derivation,greennaghdi1991,green1993thermoelasticity}, and more recently in~\cite{peshkov2018continuum,dhaouadi2024eulerian,gaybalmaz2025}.
	\end{rmrk}
	
	\subsection{Dissipative extension and Fourier limit}\label{sec:dissipative_sys}
	In this part, we extend the model \eqref{eq:heatGLM} to the case where dissipation is supplied to the system as an algebraic relaxation source term, so that the new governing equations are
	\begin{subequations}
		\begin{align}
			& \pdt{T} + c_0\, \nabla \cdot\J = 0, \label{eq:T2} \\
			& \pdt{\J} +c_0\nabla T +  c_h \rot \ppsi = -\frac{1}{\tau} \J, \label{eq:J2} \\
			& \pdt{\ppsi}- c_h\rot \J +  c_h \nabla \varphi = 0, \label{eq:psi2} \\
			& \pdt{\varphi} + c_h \nabla \cdot \ppsi = 0. \label{eq:phi2}
		\end{align}
		\label{eq:CattaneoGLM}
	\end{subequations}
	The parameter $\tau>0$ is a constant relaxation time. The subsystem (\ref{eq:T2})-(\ref{eq:J2}), omitting the curl-cleaning coupling terms, is usually referred to as the Cattaneo--Vernotte system for relativistic heat conduction \cite{cattaneo1948,cattaneo1958form,vernotte1958paradoxes}, and allows heat transfer processes to be described by a finite-speed wave motion. For the particular scaling $\tau = \kappa/c_0^2$ where $\kappa>0$ is the thermal conductivity, the Cattaneo--Vernotte system is known to be asymptotically consistent with the Fourier heat equation in the relaxation limit $c_0\to+\infty$, as in kinetic theory; see, for example, \cite{liu1987hyperbolic,Cercignani1988}. We briefly show here that this asymptotic consistency actually holds regardless of the choice of the cleaning term. We start by proving the following lemma: 
	\begin{lmm}
		\label{lmm:Fourier}
		Assume that initial data for $(T,\J)$ satisfy
		\begin{equation}
			T(\bx, 0) = T_0(\bx) \in H^3(\Omega), \qquad
			\nabla\cdot\J(\bx, 0) = \nabla\cdot\J_0(\x) = -\frac{\kappa}{c_0}\Delta T_0(\bx) + \frac{s_0(\bx)}{c_0^2},
			\label{eq:WP_IC}
		\end{equation}
		where $s_0 \in H^1(\Omega)$ is arbitrary, and the boundary conditions are such that
		\begin{equation}
			\int_{\partial \Omega} \pdt{\nabla\cdot\J}\,\pd{\nabla\cdot \J}{\bn} \, \d S =0, \ \forall t\geq0.
			\label{eq:WP_BC}
		\end{equation}
		Then, solutions of \eqref{eq:CattaneoGLM} satisfy
		\begin{equation*}
			\left\|\nabla \cdot \partial_t \J(\cdot,t)\right\|_{L^2}^2
			+ c_0^2 \left\|\nabla\,\nabla\cdot\J(\cdot,t)\right\|_{L^2}^2
			\leq 2\kappa^2 \left\|\nabla \Delta T_0\right\|_{L^2}^2 + \frac{1}{\kappa^2}\|s_0\|_{L^2}^2 + 2\|\nabla s_0\|_{L^2}^2,
		\end{equation*}
		uniformly in $c_0 \ge 1$ and $t\geq 0$. In particular, for $s_0 \equiv 0$, this reduces to 
		\begin{equation*}
			\left\|\nabla \cdot \partial_t \J(\cdot,t)\right\|_{L^2}^2
			+ c_0^2 \left\|\nabla\,\nabla\cdot\J(\cdot,t)\right\|_{L^2}^2
			\leq \kappa^2 \left\|\nabla \Delta T_0\right\|_{L^2}^2.
		\end{equation*}
	\end{lmm}
	\begin{proof}
		Taking $\pdt{} \eqref{eq:J2} - c_0\nabla \eqref{eq:T2}$  yields 
		\begin{equation*}
			\pd{^2\J}{t^2} + \frac{1}{\tau}\pdt{\J} - c_0^2\,\nabla(\nabla\cdot\J) + c_h \pdt{\rot \ppsi} = 0.
		\end{equation*}
		Applying a divergence operator to the latter and denoting $u = \nabla\cdot\J$ yields the damped wave equation
		\begin{equation*}
			\pd{^2u}{t^2}+ \frac{1}{\tau}\pdt{u} - c_0^2\,\Delta u = 0,
		\end{equation*}
		which satisfies the energy decay law
		\begin{equation*}
			\frac{\partial}{\partial t}\prn{\frac{1}{2}\prn{\pdt{u}}^2 + \frac{c_0^2}{2}\|\nabla u\|^2}
			- \nabla\cdot\prn{c_0^2\pdt{u}\nabla u} = -\frac{1}{\tau}\prn{\pdt{u}}^2 \leq 0.
		\end{equation*}
		Under boundary conditions verifying \eqref{eq:WP_BC}, one obtains the global estimates 
		\begin{equation*}
			\int_\Omega \left(\prn{\nabla\cdot\partial_t\J(\x,t)}^2 + c_0^2\,\|\nabla\,\nabla\cdot\J(\x,t)\|^2 \right) \d\Omega 
			\leq \int_\Omega \left( \prn{\nabla\cdot\partial_t\J(\x,0)}^2 + c_0^2\,\|\nabla\,\nabla\cdot\J(\x,0)\|^2 \right) \d\Omega,
		\end{equation*}
		where, evaluating $\nabla\cdot\eqref{eq:J2}$ at $t=0$ allows us to recast the right-hand side as
		\begin{align*}
			\int_\Omega \prn{ \prn{\nabla\cdot\partial_t\J(\x,t)}^2 + c_0^2\,\|\nabla\,\nabla\cdot\J(\x,t)\|^2 } \d\Omega 
			& \leq \  c_0^2\int_\Omega \prn{\Delta T(\x,0) +\frac{c_0}{\kappa}\,\nabla\cdot\J(\x,0)}^2d\Omega  \\
			& \ + c_0^2 \int_\Omega \|\nabla\,\nabla\cdot\J(\x,0)\|^2 \d\Omega.
			\label{eq:energy}
		\end{align*}
		A sufficient condition to obtain a uniform bound is to prescribe an initial condition for $\J$ verifying
		\begin{equation}
			\nabla\cdot\J(\x,0)= -\frac{\kappa}{c_0}\Delta T_0(\x) + \frac{s_0(\x)}{c_0^2}, \qquad \text{where} \quad s_0(\x) \in H^1(\Omega).
			\label{eq:IC_J}
		\end{equation}
		In this case, one obtains 
		\begin{align*}
			c_0^2\int_\Omega \prn{\Delta T_0(\x) + \frac{c_0}{\kappa}\nabla\cdot\J_0(\x)}^2\d\Omega 
			&= \frac{1}{\kappa^2} \| s_0\|_{L^2}^2, \\
			c_0^2 \int_\Omega \|\nabla\,\nabla\cdot\J_0(\x)\|^2 \d\Omega 
			&\leq 2\kappa^2\int_\Omega\|\nabla\Delta T_0(\x)\|^2\,\d\Omega 
			+ \frac{2}{c_0^2}\int_\Omega\|\nabla s_0(\x)\|^2\,\d\Omega \\
			&\leq 2\kappa^2\|\nabla\Delta T_0\|_{L^2}^2 + 2\|\nabla s_0\|_{L^2}^2,
		\end{align*}		
		so that finally
		\begin{align*}
			\nrmLp{\nabla\cdot\partial_t\J(\x,t)}^2 + c_0^2 \nrmLp{\nabla\,\nabla\cdot\J(\x,t)}^2	\leq 2\kappa^2\|\nabla\Delta T_0\|_{L^2}^2 + \frac{1}{\kappa^2} \|s_0\|_{L^2}^2 + 2\|\nabla s_0\|_{L^2}^2.
		\end{align*}
		In particular, for $s_0\equiv0$, Young's inequality is no longer required and the estimate improves to
		\begin{equation*}
			\nrmLp{\nabla\cdot\partial_t\J(\x,t)}^2 + c_0^2 \nrmLp{\nabla\,\nabla\cdot\J(\x,t)}^2
			\leq \kappa^2\|\nabla\Delta T_0\|_{L^2}^2.
		\end{equation*}
	\end{proof}
	The Fourier limit is then obtained as a direct consequence.
	\begin{prop}[Fourier limit]
		\label{prop:Fourier}
		For $\tau=\kappa/c_0^2$ and assuming well-prepared initial and boundary conditions satisfying \eqref{eq:WP_IC}-\eqref{eq:WP_BC}, smooth solutions of \eqref{eq:CattaneoGLM} 
		are asymptotically consistent with the Fourier heat equation in the sense that the residual
		\begin{equation*}
			\left\|
			\pdt{T}-\kappa\Delta T
			\right\|_{L^2(\Omega)}
			\to 0
			\qquad\text{as }\,\, c_0\to+\infty.
		\end{equation*}
	\end{prop}
	\begin{proof}		
		By taking $\kappa/c_0\div{\eqref{eq:J2}}$ and substituting $\tau = \kappa/c_0^2$, we obtain	
		\begin{equation*}
			\frac{\kappa}{c_0} \pdt{\prn{\div{\J}}} +\kappa\, \Delta T = -c_0 \div\J.
		\end{equation*}
		Substituting into \eqref{eq:T2} yields
		\begin{equation*}
			\pdt{T} - \kappa\, \Delta T = \frac{\kappa}{c_0} \pdt{\prn{\div{\J}}}.
		\end{equation*}
		By virtue of Lemma~\ref{lmm:Fourier}, there exists a constant $K\geq0$, independent of $c_0$, such that
		\begin{equation*}
			\nrmLp{\pdt{\prn{\div{\J}}}}^2
			\le
			\nrmLp{\pdt{\prn{\div{\J}}}}^2
			+c_0^2\nrmLp{\nabla\prn{\div{\J}}}^2
			\le
			K^2
		\end{equation*}
		uniformly in $c_0\ge1$ and $t\ge0$. Hence
		\begin{equation*}
			\left\|
			\pdt{T}-\nabla\cdot(\kappa\nabla T)
			\right\|_{L^2}
			\le
			\frac{\kappa}{c_0}K
			\xrightarrow{c_0\to+\infty}
			0.
		\end{equation*}
	\end{proof}
	\subsection{Symmetric hyperbolicity}
	The reversible heat-GLM system \eqref{eq:heatGLM} in three space dimensions can be cast in the form
	\begin{equation*}
		\pdt{\Q} + \sum_{i=1}^3 \B_i \pd{\Q}{x_i} = 0,
	\end{equation*}
	with $\Q = (J_1,J_2,J_3,\varphi,\psi_1,\psi_2,\psi_3,T)^\mathrm{T}$ the vector of state variables. This allows for a simple and symmetric form of the $\B_i, i\in\{1,2,3\}$, which in this case read 
	\begin{equation*}
		\setlength{\arraycolsep}{2.2pt} 
		\B_i = \begin{pmatrix}
			\mathbf{O}_{4} & \mathbf{D}_i \\
			\mathbf{D}_i^\mathrm{T} & \mathbf{O}_{4}
		\end{pmatrix}, \ \  \mathbf{D}_1 = \begin{pmatrix}
			0 & 0 & 0 & c_0\\
			0 & 0 & -c_h & 0 \\
			0 & c_h & 0 & 0 \\
			c_h & 0 & 0 & 0
		\end{pmatrix},
		\ \  \mathbf{D}_2 = \begin{pmatrix}
			0 & 0 & c_h & 0\\
			0 & 0 & 0   & c_0 \\
			-c_h  & 0 & 0 & 0 \\
			0 & c_h & 0 & 0
		\end{pmatrix},
		\ \  \mathbf{D}_3 = \begin{pmatrix}
			0 & -c_h & 0 & 0\\
			c_h & 0 & 0   & 0 \\
			0  & 0 & 0 & c_0 \\
			0 & 0 & c_h & 0
		\end{pmatrix}.
	\end{equation*}
	It follows that for every $(\beta_1,\beta_2,\beta_3)\in\RR^3$, the matrix $\B=\sum_i \beta_i \B_i$ is symmetric and hence the system of equations \eqref{eq:heatGLM} is symmetric $t$-hyperbolic in the sense of Friedrichs. The characteristic speeds in the $x$-direction, for example, are given by 
	\begin{equation*}
		\lambda _{1} = -c_0, \quad  
		\lambda _{2,3,4} = -c_h, \quad  
		\lambda _{5,6,7} = c_h,\quad  
		\lambda_{8} = c_0,
	\end{equation*}
	alongside the corresponding right eigenvectors of $\B_1$, gathered here as the columns of the matrix 
	\begin{equation*}
		\mathbf{R}_x = \left(
		\begin{array}{rrrrrrrr}
			-1 & 0 &  0 &  0 & 0 & 0 &  0 & 1 \\
			0 & 1 &  0 &  0 & 0 & 0 &  1 & 0 \\
			0 & 0 & -1 &  0 & 0 & 1 &  0 & 0 \\
			0 & 0 &  0 & -1 & 1 & 0 &  0 & 0 \\
			0 & 0 &  0 &  1 & 1 & 0 &  0 & 0 \\
			0 & 0 &  1 &  0 & 0 & 1 &  0 & 0 \\
			0 & 1 &  0 &  0 & 0 & 0 & -1 & 0 \\
			1 & 0 &  0 &  0 & 0 & 0 &  0 & 1 \\
		\end{array}
		\right).
	\end{equation*} 
	
	\section{Conserved and dissipated quantities}\label{sec:cont_constraints}
	 In this section we analyze conserved quantities of the reversible system of equations \eqref{eq:heatGLM}, some of which become adequately dissipated when the relaxation source term is added, \textit{i.e.}, for system \eqref{eq:CattaneoGLM}. We also include remarks on the resulting involutions for the system.
	\subsection{Total energy}
	Expressing the Lagrangian density $\Lambda$ from~\eqref{eq:Lagrangian} in terms of $(T, \J, \ppsi, \varphi)$ via the definitions~\eqref{def:T_psi} and~\eqref{def:J_phi}, one obtains
	\begin{equation*}
		\L(T,\J,\ppsi,\varphi) = \half \prn{T^2 + \nrm{\ppsi}^2 - \nrm{\J}^2 - \varphi^2},
	\end{equation*}
	to which corresponds the total energy density
	\begin{equation}
		\en(T,\J,\ppsi,\varphi) = \half \prn{T^2 + \nrm{\J}^2 + \nrm{\ppsi}^2 + \varphi^2},
		\label{def:en}
	\end{equation}
	whose balance law is established below.
	\begin{prop}[Non-increase of total energy]
		\label{prop:energy}
		
		The total energy
		\begin{equation*}
			\toten(t)
			=
			\int_{\Omega}
			\frac12\Bigl(
			T(\x,t)^2
			+\nrm{\J(\x,t)}^2
			+\nrm{\ppsi(\x,t)}^2
			+\varphi(\x,t)^2
			\Bigr)\,\d\Omega
		\end{equation*}
		satisfies, up to boundary terms,
		\begin{equation}
			\toten(t)-\toten(0)
			=
			-\frac1{\tau}
			\int_0^t \int_{\Omega}
			\nrm{\J(\x,s)}^2 \,\d\Omega\,\d s
			\le 0.
			\label{eq:en_balance}
		\end{equation}
		
		In particular:
		\begin{enumerate}
			\item[(a)] For finite $\tau>0$, the energy is non-increasing.
			\item[(b)] In the absence of relaxation $(1/\tau=0)$, the total energy is conserved.
		\end{enumerate}
		Consequently, the solution is uniformly bounded in
		$L^\infty\prn{0,t_f;L^2(\Omega)}$.
	\end{prop}
	
	\begin{proof}
		One first derives the evolution equation for the total energy density \eqref{def:en}, obtained by summing each of the equations (\ref{eq:T2}-\ref{eq:phi2}) multiplied by the corresponding dual or main-field variables $\mathbf{P} = \partial \mathcal{E} / \partial \mathbf{Q}$ (here equal to the vector of state variables), to obtain
		\begin{equation}
			\pdt{\en} + \nabla \cdot \prn{c_0\, T \J + c_h\, \ppsi \times \J + c_h \,\varphi\, \ppsi} = -\frac{1}{\tau} \nrm{\J}^2 \le 0.
			\label{eq:en_lyapunov}
		\end{equation}
		Integrating both sides over the domain $\Omega$ yields~\eqref{eq:en_balance}.
	\end{proof}
	\subsection{Rotational energy and curl involution}
	Let $\curl_\J = \rot \J$ and $\curl_{\ppsi} = \rot \ppsi$. We call the quantity
	\begin{equation*}
		e_{\curl} = \frac{1}{2} \prn{\nrm{\curl_\J}^2 + \nrm{\curl_{\ppsi}}^2},
	\end{equation*}
	the rotational energy density of the hyperbolic heat-GLM system, for which we prove the following.
	\begin{prop}[Non-increase of total rotational energy]
		\label{prop:curl_energy}
		The quantity $e_{\curl}$ obeys the balance law
		\begin{equation}
			\pdt{e_{\curl}}	+ \div\!\left(c_h\, \curl_{\ppsi} \times \curl_\J\right)		= -\frac{1}{\tau}\nrm{\curl_\J}^2 \leq 0,
			\label{eq:E_omega}
		\end{equation}
		As a result, the total rotational energy 	
		\begin{equation*}
			\toten_{\curl} = \int_\Omega \left( \frac{1}{2}\nrm{\curl_\J}^2 + \frac{1}{2}\nrm{\curl_{\ppsi}}^2 \right) \d\Omega,
		\end{equation*}
		satisfies, up to boundary terms,
		\begin{equation}
			\toten_{\curl}(t)-\toten_{\curl}(0)
			=
			-\frac1{\tau}
			\int_0^t \int_{\Omega}
			\nrm{\curl_\J(\x,s)}^2 \,\d\Omega\,\d s
			\le 0.
			\label{eq:E_omega_integral}
		\end{equation}
		In particular:
		\begin{enumerate}
			\item[(a)] For finite $\tau>0$, $\toten_{\curl}$ is non-increasing.
			\item[(b)] In the absence of relaxation $(1/\tau=0)$, $\toten_{\curl}$ is conserved. 
		\end{enumerate}
		As a consequence, the following involution holds for boundary conditions such that boundary terms vanish:
		\begin{align*}
			\text{if} \quad \begin{cases}
				\displaystyle
				\rot\J(\x,0)=0,\\
				\displaystyle
				\rot\ppsi(\x,0)=0,
			\end{cases}
			\quad \text{then} \qquad 
			\begin{cases}
				\displaystyle
				\rot\J(\x,t)=0,\\
				\displaystyle
				\rot\ppsi(\x,t)=0,
			\end{cases}
			\qquad \forall\, t>0.
		\end{align*}
	\end{prop}
	\begin{proof}
		Applying the curl operator to \eqref{eq:J2} and \eqref{eq:psi2} yields
		\begin{subequations}
			\begin{align}
				&\pdt{\curl_\J} + c_h\, \rot \curl_{\ppsi} = - \frac{1}{\tau} \curl_\J, \label{eq:curlJ_t}\\
				&\pdt{\curl_{\ppsi}} - c_h\, \rot \curl_{\J} = 0.\label{eq:curlpsi_t}
			\end{align}
			\label{eq:curl_evolution}
		\end{subequations}
		Multiplying~\eqref{eq:curlJ_t} by $\curl_\J$ and~\eqref{eq:curlpsi_t} by $\curl_{\ppsi}$, and summing them results in~\eqref{eq:E_omega}. Integrating the latter over space yields \eqref{eq:E_omega_integral}. 	
		\end{proof}
		\begin{rmrk}
			One can decouple the equations \eqref{eq:curl_evolution} by taking $\pdt{}\eqref{eq:curlJ_t} - c_h \nabla\times \eqref{eq:curlpsi_t}$ and $\pdt{}\eqref{eq:curlpsi_t} + c_h \nabla\times \eqref{eq:curlJ_t}$, which leads to both $\curl_\J$ and $\curl_{\ppsi}$ satisfying the same damped wave equation 
			\begin{equation*} 
			\pd{^2\curl}{t^2} + \frac1\tau \pd{\curl}{t}- c_h^2\, \nabla^2 \curl = 0. 
			\end{equation*} 
		\end{rmrk}
	
		\subsection{Wave-equation energies}\label{sec:waveeq_energy}
		Lastly, we show that the system \eqref{eq:CattaneoGLM} admits additional conserved/dissipated energies associated with the wave structure of the pairs $(T,\nabla\cdot\J)$ and $(\varphi,\nabla\cdot\ppsi)$. Indeed, this property arises as a direct consequence of the fact that $T$ and $\varphi$ satisfy, respectively, a damped and a classical wave equation
		\begin{subequations}
			\begin{align}
				\pd{^2 T}{t^2} + \frac{1}{\tau}\pdt{T} - c_0^2 \Delta T &= 0,
				\label{eq:T_wave}
				\\
				\pd{^2 \varphi}{t^2} - c_h^2 \Delta \varphi &= 0,
				\label{eq:phi_wave}
			\end{align}
		\end{subequations}
		obtained by taking $\pdt{}\eqref{eq:T2} - c_0 \nabla\cdot \eqref{eq:J2}$ and $\pdt{}\eqref{eq:phi2} - c_h \nabla\cdot \eqref{eq:psi2}$, respectively. It follows naturally that the associated (scaled) energy densities
		\begin{equation*}
			e_T = \frac{1}{2}\prn{\frac{1}{c_0^2}\prn{\pdt{T}}^2+\nrm{\nabla T}^2},  \quad 
			e_\varphi = \frac{1}{2}\prn{\frac{1}{c_h^2}\prn{\pdt{\varphi}}^2+\nrm{\nabla \varphi}^2}, 
		\end{equation*}
		satisfy the balance laws
		\begin{align*}
			&\pdt{e_T}-\div\prn{\pdt{T}\,\nabla T}	
			=-\frac{1}{\tau c_0^2}\prn{\pdt{T}}^2, \\
			&\pdt{e_\varphi}	-\div\prn{\pdt{\varphi}\,\nabla \varphi}
			=0.
		\end{align*}
		Consequently, we obtain the following. 
		\begin{prop}[Wave-equation energies]\label{prop:wave_inv}
			Define the total wave energies
			\begin{equation*}
				\toten_T = \int_\Omega \frac{1}{2}\prn{\nrm{\nabla T}^2 + \prn{\nabla\cdot\J}^2} \d\Omega, \quad 
				\toten_\varphi 	= \int_\Omega \frac{1}{2}\prn{\nrm{\nabla \varphi}^2 + \prn{\nabla\cdot\ppsi}^2} \d\Omega.
			\end{equation*}
			Then, smooth solutions of \eqref{eq:CattaneoGLM} satisfy, up to boundary terms
			\begin{subequations}
				\begin{align}
					&\toten_T(t)-\toten_T(0) = -\frac1{\tau} \int_0^t \int_{\Omega}
					\prn{\nabla\cdot\J(\x,s)}^2 \,\d\Omega\,\d s \le 0. \label{eq:E_T_integral} \\
					&\toten_\varphi(t)-\toten_\varphi(0) = 0. \label{eq:E_phi_integral}
				\end{align}
			\end{subequations}
			In particular:
			\begin{enumerate}
				\item[(a)] For finite $\tau>0$, $\toten_T$ is non-increasing and $\toten_\varphi$ is conserved. 
				\item[(b)] In the absence of relaxation $(1/\tau=0)$, both $\toten_T$ and $\toten_\varphi$ are conserved. 
			\end{enumerate}
			As a consequence, in both cases the following involutions are satisfied for boundary conditions such that boundary terms vanish:
			\begin{subequations}
				\begin{align}
					&\text{if} \quad
					\begin{cases}
						\displaystyle
						\nabla T(\x,0)=0,\\
						\displaystyle
						\nabla\cdot\J(\x,0)=0,
					\end{cases}
					\quad \text{then} \qquad
					\begin{cases}
						\displaystyle
						\nabla T(\x,t)=0,\\
						\displaystyle
						\nabla\cdot\J(\x,t)=0,
					\end{cases}
					\qquad \forall\, t>0, \label{eq:div_J}\\
					\text{and} \ &\text{if} \quad
					\begin{cases}
						\displaystyle
						\nabla \varphi(\x,0)=0,\\
						\displaystyle
						\nabla\cdot\ppsi(\x,0)=0,
					\end{cases}
					\quad \text{then} \qquad
					\begin{cases}
						\displaystyle
						\nabla \varphi(\x,t)=0,\\
						\displaystyle
						\nabla\cdot\ppsi(\x,t)=0,
					\end{cases}
					\qquad \forall\, t>0. \label{eq:div_psi}
				\end{align}
			\end{subequations}
		\end{prop}
		\begin{proof}
			By virtue of \eqref{eq:T2}, we can rewrite both $e_T$ and its associated balance law as 
			\begin{equation*}
				e_T = \frac{1}{2}\prn{\nrm{\nabla T}^2 + \prn{\nabla\cdot\J}^2}, \qquad \pdt{e_T}+\div\prn{c_0\prn{\nabla\cdot\J}\nabla T}	
				=-\frac{1}{\tau}\prn{\nabla\cdot\J}^2.
			\end{equation*}
			Integrating the latter over space yields \eqref{eq:E_T_integral}. Therefore, if the boundary conditions preserve \eqref{eq:E_T_integral} and $e_T=0$ initially, then $e_T$ vanishes for all later times, resulting in the involution \eqref{eq:div_J}. Analogous considerations lead to the same results for $(\varphi,\ppsi)$.      
		\end{proof}
		
		\section{Numerical scheme: notations and definitions}\label{sec:num_discretization}
		We denote by $\Omega$ the computational domain and consider a staggered mesh configuration. The primal mesh consists of control volumes $\Omega_c$ with cell centers $\x_c$, while its dual consists of control volumes $\Omega_p$ with centers $\x_p$, obtained by connecting neighboring primal cell centers, as illustrated in Figure~\ref{fig:staggered_mesh}. We further denote the set of cell centers by $\V_c$ and the set of nodes by $\V_p$, with cardinalities $N_c$ and $N_p$, respectively. Such a configuration allows the construction of compatible discrete operators, following the classical setting presented in \cite{HymanShashkov1997,Maire2007,Maire2008,Maire2009}. 
		
		\begin{figure}[!h]
			\centering
			\includegraphics[width=0.33\textwidth]{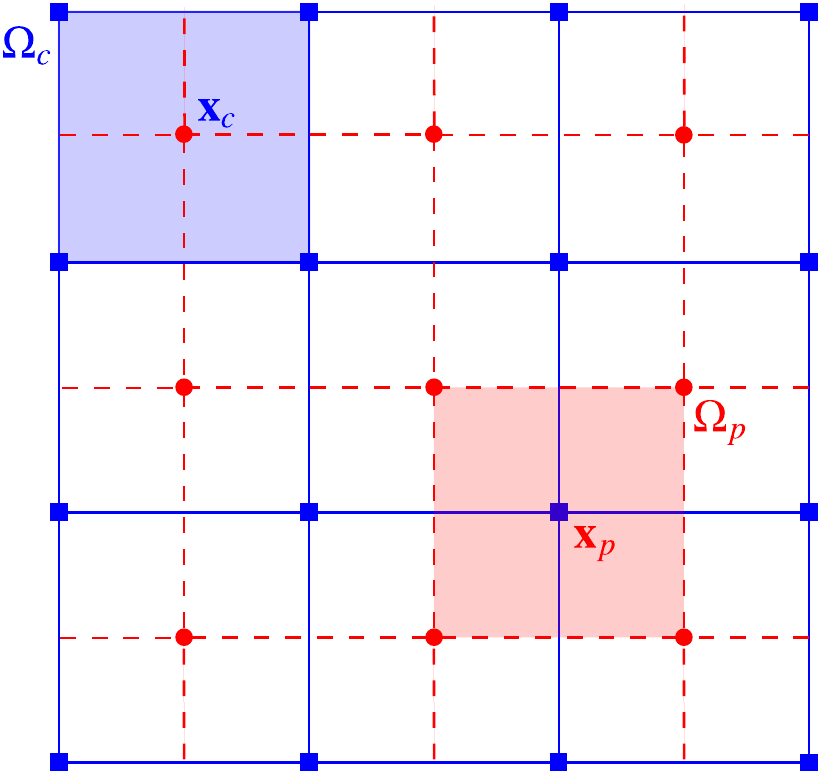}
			\caption{Sketch of the staggered grid, showing primal cells $\Omega_c$ (blue) with centers $\x_c$ and dual cells $\Omega_p$ (red) with centers $\x_p$.}
			\label{fig:staggered_mesh}
		\end{figure}
		At the core of these methods lie discrete nabla operators defined as weighted sums of corner vectors over the respective cell boundaries:
		\begin{align}
			&\nablacp \circ = \frac{1}{|\Omega_c|} \sum_{p \in \Omega_c} l_{pc} \bn_{pc} \circ, \label{eq:nablacp}  \\
			&\nablapc \circ = \frac{1}{|\Omega_p|} \sum_{c \in \Omega_p} l_{pc} \bn_{cp} \circ,	\label{eq:nablapc} 
		\end{align} 	
		where $\bn_{cp} = -\bn_{pc}$. For each pair $(c,p)$ composed of a primal cell $\Omega_c$ and one of its vertices $\mathbf{x}_p$, we define the dual subcell $\Gamma_{pc} = \partial \Omega_p \cap \partial \Omega_{cp},$ where $\Omega_{cp}$ denotes the subcell associated with the pair $(c,p)$, obtained by joining the center of the primal cell $\bx_c$, the vertex $\bx_p$, and the midpoints of the edges of $\Omega_c$ incident to $p$. The quantity $l_{pc}\,\bn_{pc}$ is the \textit{corner vector} associated with the vertex $p$ and the primal cell $\Omega_c$, defined as
		\begin{equation*}
			l_{pc}\bn_{pc}=\int_{\Gamma_{pc}} \bn \, \mathrm{d} \ S ,
		\end{equation*}
		where $l_{pc}$ is the length of the segment $\Gamma_{pc}$ and $\bn_{pc}$ is the unit normal vector to $\Gamma_{pc}$ pointing outward from the dual cell $\Omega_p$. On a Cartesian mesh the corner vectors are simply $\halb (+\Delta y, +\Delta x)$, $\halb (-\Delta y, +\Delta x)$, $\halb (-\Delta y, -\Delta x)$ and $\halb (+\Delta y, -\Delta x)$, starting from the upper right corner vector and using counter-clockwise ordering.  
		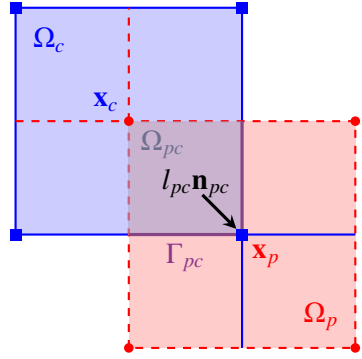
\begin{figure}[!h]
			\centering
			\begin{tikzpicture}[scale=1.4,>=stealth]
				\fill[blue!20] (0,0) rectangle (2,2);
				\draw[thick,blue] (0,0) rectangle (2,2);
				\node[blue] at (0.3,1.72) {$\Omega_c$};
				\node[blue,above left] at (1,1) {$\mathbf{x}_c$};
				\fill[red!20] (1,-1) rectangle (3,1);
				\draw[thick,red,dashed] (1,-1) rectangle (3,1);
				\node[red] at (2.7,-0.72) {$\Omega_p$};
				\node[red,below right] at (2,0) {$\mathbf{x}_p$};
				\fill[Periwinkle!60,opacity=0.6] (1,1) -- (1,0) -- (2,0) -- (2,1) -- cycle;
				\node[Periwinkle!60!black] at (1.3,0.8) {$\Omega_{pc}$};
				\draw[very thick,Plum!80!black] (1,0) -- (2,0) -- (2,1);
				\node[Plum!90!black] at (1.5,-0.2) {$\Gamma_{pc}$};
				\draw[->,very thick] (1.65,0.35) -- (1.95,0.05);
				\node[right] at (1.2,0.45) {$l_{pc}\mathbf n_{pc}$};
				\draw[blue,thick] (2,0) -- (3,0);
				\draw[blue,thick] (2,-1) -- (2,0);
				\draw[red,dashed,thick] (1,1) -- (1,2);
				\draw[red,dashed,thick] (0,1) -- (1,1);
				\fill[blue] (-0.05,-0.05) rectangle (0.05,0.05);
				\fill[blue] (-0.05,1.95) rectangle (0.05,2.05);
				\fill[blue] (1.95,-0.05) rectangle (2.05,0.05); 
				\fill[blue] (1.95,1.95) rectangle (2.05,2.05);
				\fill[red] (1,1) circle (1.2pt);	
				\fill[red] (1,-1) circle (1.2pt);	
				\fill[red] (3,1) circle (1.2pt);	
				\fill[red] (3,-1) circle (1.2pt);	
			\end{tikzpicture}
			\caption{Geometric meaning of $l_{pc}\mathbf{n}_{pc}$. The blue cell is the primal control volume $\Omega_c$. The red dashed cell is the dual control volume $\Omega_p$, centered at the primal vertex $\mathbf{x}_p$. $l_{pc}$ is the total length of the segment $\Gamma_{pc}$, and $\mathbf{n}_{pc}$ is the corresponding unit normal vector.}
			\label{fig:lpcnpc}
		\end{figure}
		
		Since the vectors $\{l_{pc}\bn_{pc}\}_{p\in\Omega_c}$ form a closed polygon, it follows immediately from the Gauss theorem that the following identities hold:
		\begin{equation*}
			\sum_{p \in \Omega_c} l_{pc} \bn_{pc} = 0, \qquad 
			\sum_{c \in \Omega_p} l_{pc} \bn_{cp} = 0. 
		\end{equation*} 
		Let $\phi_p$ and $\phi_c$ be scalar fields defined at the vertices and centers of the primal mesh, respectively, and let $\A_p$ and $\A_c$ be vector fields defined at the vertices and centers of the primal mesh, respectively. We can define, using the definitions \eqref{eq:nablacp} and \eqref{eq:nablapc}, the following compatible discrete gradient, divergence, and curl operators:
		\begin{subequations}
			\begin{alignat}{2}
				&\nablacp \phi_p = \frac{1}{|\Omega_c|} \sum_{p \in \Omega_c} l_{pc} \bn_{pc}\, \phi_p, 
				&&\nablapc \phi_c = \frac{1}{|\Omega_p|} \sum_{c \in \Omega_p} l_{pc} \bn_{cp}\, \phi_c, \label{eq:compatible_grad} \\  	
				&\divcp \A_p = \frac{1}{|\Omega_c|} \sum_{p \in \Omega_c} l_{pc} \bn_{pc} \cdot \A_p, 
				&&\divpc \A_c = \frac{1}{|\Omega_p|} \sum_{c \in \Omega_p} l_{pc} \bn_{cp} \cdot \A_c, \label{eq:compatible_div} \\ 
				&\rotcp \A_p = \frac{1}{|\Omega_c|} \sum_{p \in \Omega_c} l_{pc} \bn_{pc} \times \A_p, \qquad 
				&&\rotpc \A_c = \frac{1}{|\Omega_p|} \sum_{c \in \Omega_p} l_{pc} \bn_{cp} \times \A_c. \label{eq:compatible_curl}
			\end{alignat} 
		\end{subequations}
		One can then verify that the continuous vector calculus identities $\rot \nabla \phi = 0$ and $\div \rot \A = 0$ also hold at the discrete level \cite{dumbser2025maxwellGLM}, as a direct consequence of the formulas~(\ref{eq:compatible_grad})--(\ref{eq:compatible_curl}), so that 
		\begin{subequations}
			\begin{alignat}{2}
				&\rotcp \nablapc \, \phi_c = 0, \quad 
				&&\rotpc \nablacp \, \phi_p = 0,
				\label{eq:curlgrad} \\
				&\divcp \rotpc \A_c = 0, \quad 
				&&\divpc \rotcp \A_p = 0.
				\label{eq:divcurl}
			\end{alignat}
		\end{subequations}
		These discrete identities will be the key ingredient
		in proving the main results of Section~\ref{sec:num_discretization}.
		The discrete heat flux and the discrete cleaning scalar are located at the cell centers of the primal mesh $\J_c^n=\J(\x_c,t^n)$ and $\varphi_c^n=\varphi(\x_c,t^n)$. The discrete temperature and the discrete cleaning vector are located at the vertices of the primal mesh (cell centers of the dual mesh)
		$T_p^n=T(\x_p,t^n)$ and $\ppsi_p^n=\ppsi(\x_p,t^n)$. Furthermore, for any field $q$, we define the temporal midpoint rule  by 
		\begin{equation*}
			q^{{n+\half}} = \dfrac{1}{2} \left(q^{n}+q^{n+1}\right)
		\end{equation*}
		
		\subsection{Discretization of the reversible system}
		We can discretize the reversible system~\eqref{eq:heatGLM} as follows
		\begin{subequations}
			\begin{align}
				& T_p^{n+1} =T_p^{n}-c_0\,\Delta t \, \divpc\J_c^{n+\half} , \label{eq:T-disc} \\
				& \J_c^{n+1} = \J_c^{n}-c_0\,\Delta t \,\nablacp T_p^{n+\half} - c_h\,\Delta t \, \rotcp \ppsi_p^{n+\half} , \label{eq:J-disc} \\
				& \ppsi_p^{n+1} = \ppsi_p^{n}+c_h\,\Delta t \,\rotpc \J_c^{n+\half} - c_h\,\Delta t \, \nablapc \varphi_c^{n+\half}, \label{eq:psi-disc} \\
				& \varphi_c^{n+1} = \varphi_c^{n} - c_h\,\Delta t \,\divcp \ppsi_p^{n+\half}. \label{eq:phi-disc}
			\end{align}
			\label{eq:heatGLM-disc}
		\end{subequations} 
		Equation~\eqref{eq:J-disc} will be written as
		\begin{equation}
			\J_c^{n+1} = \Js-\frac{c_0\,\Delta t}{2}\nablacp T_p^{n+1} - c_h\,\Delta t \, \rotcp \ppsi_p^{n+\half}, \quad \mbox{with}\,\, \Js = \J_c^{n}-\frac{c_0\,\Delta t}{2}\nablacp T_p^{n}.\label{eq:Jst-disc} 
		\end{equation}
		Inserting~\eqref{eq:Jst-disc} into~\eqref{eq:T-disc} and using the discrete vector calculus identity~\eqref{eq:divcurl}, one obtains the following discrete wave equation for the scalar $T$: 
		\begin{equation}
			T_p^{n+1} - \frac{c_0^2\Delta t^2}{4}\divpc \nablacp T_p^{n+1} =  T_p^{n} - \frac{c_0\,\Delta t}{2} \divpc \left(\J_c^{n} +\Js\right).
			\label{eq:Tst-wave}
		\end{equation}
		%
		On the other hand, Equation~\eqref{eq:psi-disc} can be written as 
		\begin{equation}
			\ppsi_p^{n+1} = \ppsis+c_h\,\Delta t \,\rotpc \J_c^{n+\half} - \frac{c_h\,\Delta t}{2} \nablapc \varphi_c^{n+1}, \quad \mbox{with }\,\,\ppsis = \ppsi_p^{n} - \frac{c_h\,\Delta t}{2}\nablapc \varphi_c^{n}.\label{eq:psist-disc}
		\end{equation}
		Inserting \eqref{eq:J-disc} and \eqref{eq:phi-disc} into \eqref{eq:psist-disc} and taking into account the identity \eqref{eq:curlgrad}, we obtain a discrete vector wave equation for the field $\ppsi$
		\begin{multline}
			\ppsi_p^{n+1} + \frac{c_h^2\Delta t^2}{4} \,  \rotpc \rotcp \ppsi_p^{n+1} - \frac{c_h^2\Delta t^2}{4} \nablapc \divcp \ppsi_p^{n+1} = \\
			\ppsis + c_h\,\Delta t \, \rotpc \J_c^{n} - \frac{c_h^2\Delta t^2}{4} \rotpc \rotcp \ppsi_p^{n} - \frac{c_h\,\Delta t}{2}  \nablapc \varphi_c^{n}+ \frac{c_h^2\Delta t^2}{4} \nablapc \divcp \ppsi_p^{n}. 
			\label{eq:psist-wave} 
		\end{multline} 
		The final equations to be solved are discrete second-order wave equations for $T$ and $\ppsi$. This is not a coincidence: they are the discrete equivalents of the Euler--Lagrange equations~\eqref{eq:EL-Z} and~\eqref{eq:EL-A}, written not for the original potentials $Z$ and $\A$ but for their time derivatives $T$ and $\ppsi$, respectively. Once the scalar $T_p^{n+1}$ and the cleaning field $\ppsi_p^{n+1}$ have been computed, the heat flux $\J$ and the cleaning scalar $\varphi$ are updated via~\eqref{eq:J-disc} and \eqref{eq:phi-disc}. This methodology is analogous to the post-projection stage proposed in~\cite{harlow1965,patankar1980,karki1989,Casulli1990,CasulliCheng1992} for incompressible Navier--Stokes and shallow water equations, where the pressure and velocity are similarly decoupled through an elliptic projection step.
		
		\subsection{Discretization of the dissipative system}
		The dissipative system of equations \eqref{eq:CattaneoGLM} can be discretized analogously, as follows
		\begin{subequations}\label{eq:CattaneoGLM-disc}
			\begin{align}
				& T_p^{n+1} =T_p^{n}-c_0\,\Delta t \, \divpc\J_c^{n+\half} , \label{eq:T2-disc} \\
				& \J_c^{n+1} = \J_c^{n}-c_0\,\Delta t \,\nablacp T_p^{n+\half} - c_h\,\Delta t \, \rotcp \ppsi_p^{n+\half} -\Delta t\dfrac{\J_c^{n+\half}}{\tau}, \label{eq:J2-disc} \\
				& \ppsi_p^{n+1} = \ppsi_p^{n}+c_h\,\Delta t \,\rotpc \J_c^{n+\half} - c_h\,\Delta t \, \nablapc \varphi_c^{n+\half}, \label{eq:psi2-disc} \\
				& \varphi_c^{n+1} = \varphi_c^{n} - c_h\,\Delta t \,\divcp \ppsi_p^{n+\half}. \label{eq:phi2-disc}
			\end{align}
		\end{subequations}  
		We can rewrite~\eqref{eq:J2-disc} as
		\begin{align}
			\J_c^{n+1} &= \dfrac{2\tau-\Delta t}{2\tau+\Delta t}\J_c^{n}-c_0\,\Delta t\dfrac{2\tau}{2\tau+\Delta t}\nablacp T_p^{n+\half} - c_h\,\Delta t \dfrac{2\tau}{2\tau+\Delta t} \rotcp \ppsi_p^{n+\half}\nonumber\\
			&= \Js-c_0\,\Delta t\dfrac{\tau}{2\tau+\Delta t}\nablacp T_p^{n+1} - c_h\,\Delta t \dfrac{2\tau}{2\tau+\Delta t}\rotcp \ppsi_p^{n+\half},
			\label{eq:J2-disc2}
		\end{align}
		with 
		\begin{equation*}
			\Js = \dfrac{2\tau-\Delta t}{2\tau+\Delta t}\J_c^{n}-c_0\,\Delta t \dfrac{\tau}{2\tau+\Delta t}\nablacp T_p^{n}. 
		\end{equation*}
		Inserting~\eqref{eq:J2-disc2} into~\eqref{eq:T2-disc} and taking into account the discrete vector identity~\eqref{eq:divcurl}, $\divpc~\rotcp~\ppsi_p~= 0$, one obtains the following discrete equation for the scalar $T$:
		\begin{equation}
			T_p^{n+1} - \frac{\Delta t^2 c_0^2}{2} \dfrac{\tau}{2\tau+\Delta t}\divpc \nablacp T_p^{n+1} =  T_p^{n} - \frac{\Delta t \,c_0}{2}  \divpc \left(\J_c^{n} +\Js\right).
			\label{eq:T2-eq} 
		\end{equation}
		On the other hand, Equation~\eqref{eq:psi2-disc} can be written as
		\begin{equation}
			\ppsi_p^{n+1} = \ppsis+\Delta t \,c_h\rotpc \J_c^{n+\half} - \frac{\Delta t}{2} c_h\nablapc \varphi_c^{n+1}, \quad \mbox{with }\,\,\ppsis = \ppsi_p^{n} - \frac{\Delta t}{2} c_h\nablapc \varphi_c^{n}. \label{eq:psisst-disc}
		\end{equation}
		
		Inserting \eqref{eq:J2-disc2} and \eqref{eq:phi2-disc} into \eqref{eq:psisst-disc} and taking into account the identity \eqref{eq:curlgrad}, $\rotpc \nablacp \, T_p = 0$, we obtain a discrete vector equation for the field $\ppsi$
		\begin{align}
			& \ppsi_p^{n+1} + \frac{\Delta t^2}{2} \, c_h^2 \dfrac{\tau}{2\tau+\Delta t} \rotpc \rotcp \ppsi_p^{n+1} - \frac{\Delta t^2}{4}  \, c_h^2 \nablapc \divcp \ppsi_p^{n+1} \nonumber  \\ 
			\quad = \ &  \ppsis + \Delta t \, c_h \dfrac{2\tau}{2\tau+\Delta t}\rotpc \J_c^{n} - \frac{\Delta t^2}{2} \, c_h^2\dfrac{\tau}{2\tau+\Delta t} \rotpc \rotcp \ppsi_p^{n} - \nonumber \frac{\Delta t}{2} \, c_h \nablapc \varphi_c^{n} \\ 
			& + \frac{\Delta t^2}{4} \, c_h^2 \nablapc \divcp \ppsi_p^{n}. 
			\label{eq:psisst-wave} 
		\end{align} 
		
		As for the reversible system, once the new scalar $T_p^{n+1}$ and the new cleaning field $\ppsi_p^{n+1}$ have been computed, the heat flux $\J$ and the cleaning scalar $\varphi$ are updated via~\eqref{eq:J2-disc} and \eqref{eq:phi2-disc}.
		
		\subsection{Structure-preservation properties}
		In the following part, we show that the properties of the continuous system, proven in Sections~\ref{sec:gov_eq} and \ref{sec:cont_constraints}, also hold at the discrete level. As a preamble to these proofs, we first establish the following result: 
		\begin{prp}[Discrete Gauss--Ostrogradsky theorem]
			\label{prop:Gauss}
			For any discrete scalar field $\phi_c$ and vector field  $\B_c$ located at the cell centers, and any discrete vector field $\A_p$ located on the dual grid, we have the following global identities
			\begin{subequations}
				\begin{gather}
					\sum_{p\in\Omega} |\Omega_p| \,\,\A_p\cdot \nablapc \phi_c + \sum_{c\in\Omega} |\Omega_c| \,\,\phi_c\,\, \divcp \A_p = 0, 
					\label{eq:Gauss_divgrad} \\
					\sum_{p\in\Omega} |\Omega_p|\,\prn{ \rotpc \A_c} \cdot \B_p - \sum_{c\in\Omega} |\Omega_c|\,\prn{ \rotcp \B_p} \cdot \A_c = 0,
					\label{eq:Gauss_rot} 
				\end{gather}
			\end{subequations}
			under periodic boundary conditions. Analogous identities hold when exchanging $c$ and $p$. 
		\end{prp}
		\begin{proof}
			We start by showing the identity \eqref{eq:Gauss_divgrad}. We multiply \eqref{eq:compatible_grad} by $\phi_c$ and dot multiply \eqref{eq:compatible_div} by $\A_p$ to obtain  
			\begin{gather*}
				\A_p\cdot \nablapc \phi_c = \frac{1}{|\Omega_p|} \sum_{c \in \Omega_p} \phi_c\,l_{pc} \bn_{cp} \cdot \mathbf{A}_p, \\
				\phi_c \,\divcp \A_p = \frac{1}{|\Omega_c|} \sum_{p \in \Omega_c} \phi_c\, l_{pc} \bn_{pc} \cdot \A_p.
			\end{gather*}
			This implies in particular that the discrete integral over the computational domain of both equalities in their respective cells yields  
			\begin{equation}
				\sum_{p\in\Omega} |\Omega_p| \,\A_p\cdot \nablapc \phi_c + 	\sum_{c\in\Omega} |\Omega_c| \,\,\phi_c \divcp \A_p = \sum_{p\in\Omega} \sum_{c \in \Omega_p} l_{pc} \bn_{cp}\, \phi_c \cdot \mathbf{A}_p  + \sum_{c\in\Omega} \sum_{p \in \Omega_c} l_{pc} \bn_{pc} \cdot \A_p \,\phi_c.
				\label{eq:sum_divgrad}
			\end{equation}
			Under periodic boundary conditions, the discrete domain has no boundary contribution, and the connectivity relation between primal and dual cells is symmetric. Therefore, the order of summation can be exchanged
			\begin{equation*}
				\sum_{p\in\Omega} \sum_{c \in \Omega_p} \, (\cdot) = \sum_{c\in\Omega} \sum_{p \in \Omega_c} \, (\cdot).
			\end{equation*}
			Using this identity and recalling that $\bn_{cp} = -\bn_{pc}$ shows that the sum \eqref{eq:sum_divgrad} vanishes 
			\begin{equation*}
				\sum_{p\in\Omega} |\Omega_p| \A_p\cdot \nablapc \phi_c + 	\sum_{c\in\Omega} |\Omega_c| \phi_c \divcp \A_p = 0,
			\end{equation*}
			which is a discrete equivalent of the Gauss--Ostrogradsky theorem, 
			\begin{equation*}
				\int_\Omega \prn{\phi \nabla\cdot\A + \nabla\phi\cdot\A} \, \d\Omega = \int_\Omega \nabla\cdot\prn{\phi\A} \, \d\Omega = \int_{\partial\Omega} \phi\A\cdot\bn \, \d S, 
			\end{equation*}
			which also vanishes identically at the continuous level for periodic boundary conditions. An analogous argument establishes \eqref{eq:Gauss_rot}. To this end, we take the definitions \eqref{eq:compatible_curl} of $\rotpc\A_c$ and $\rotcp\B_p$, dot multiplied by $\B_p$ and $\A_c$, respectively
			\begin{align*}
				\prn{ \rotpc \A_c} \cdot \B_p &= \frac{1}{|\Omega_p|} \sum_{c \in \Omega_p} l_{pc} (\bn_{cp} \times \A_c)\cdot \B_p  \\
				\prn{ \rotcp \B_p} \cdot \A_c &= \frac{1}{|\Omega_c|} \sum_{p \in \Omega_c} l_{pc} (\bn_{pc} \times \B_p)\cdot \A_c = \frac{-1}{|\Omega_c|} \sum_{p \in \Omega_c} l_{pc} (\bn_{pc} \times \A_c)\cdot \B_p.
			\end{align*}
			Summing both terms over their respective domains and following the same methodology as above, one obtains \eqref{eq:Gauss_rot}. 
		\end{proof}

		\subsubsection{Property 1: Asymptotic-preservation of the Fourier limit}\label{sec:fourier-limit}
		In what follows, $\mathbf T^{n} := (T_p^{n})_p \in \mathbb R^{N_p}$ and $\mathbf J^{n} := (\J_c^{n})_c \in \mathbb R^{N_c}$
		denote, respectively, the discrete temperature and heat-flux fields at time $t^n$, and $\nrm{(\cdot)_p}_{L^2} := \sqrt{\sum_{p\in\Omega} |\Omega_p|\,(\cdot)_p^2}$
		denotes the discrete $L^2$ norm. 
		\begin{theorem} \label{thm:AP}
			For well-prepared initial data satisfying
			$\nrm{\mathbf T^0}_{L^2} = \bigO(1)$ and $\nrm{\divpc \mathbf J^0}_{L^2} = \bigO(c_0^{-m})$
			with $m \in (-1,+\infty]$, and for a fixed value of $c_h$, the
			scheme~\eqref{eq:CattaneoGLM-disc} is asymptotic-preserving in the
			sense that the discrete heat equation with a discrete Fourier law is recovered in the limit
			$c_0 \to +\infty$, \textit{i.e.},
			\begin{equation}
				\nrm{\frac{\mathbf T^{n+1}-\mathbf T^n}{\Delta t}
					- \kappa\, \mathbf L\, \mathbf T^{n+\half}}_{L^2} = 
				\bigO\prn{c_0^{-r}},
				\label{eq:discrete_residual}
			\end{equation}
			where $r = \min\prn{2,\, 1+m}$, and $\mathbf L \in \mathbb R^{N_p\times N_p}$ denotes
			the discrete global Laplace operator, defined by $\prn{\mathbf L\,\mathbf T}_p := \prn{\divpc\nablacq T_q}_p$, where the index $q$ runs, like $p$, over the vertices $\V_p$ of the primal mesh and repeated indices are summed over, in the manner of the Einstein summation convention.
		\end{theorem}	
		\begin{proof}
			We start from the discrete equation \eqref{eq:J2-disc} which we rewrite as
			\begin{equation*}
				\tau \frac{\J_c^{n+1} - \J_c^{n}}{\Delta t} + \J_c^{n+\half} = -c_0 \tau \,\nablacp T_p^{n+\half} - c_h \tau \, \rotcp \ppsi_p^{n+\half}.
			\end{equation*}
			Note that 
			\begin{equation*}
				\frac{\J_c^{n+1} - \J_c^{n}}{\Delta t} = \frac{\J_c^{n+1} + \J_c^{n} -\J_c^{n} -  \J_c^{n}}{\Delta t}  = 2\frac{\J_c^{n+\half} - \J_c^{n}}{\Delta t},
			\end{equation*}
			which gives
			\begin{equation*}
				\prn{1+\dfrac{2\tau}{\Delta t}}\J_c^{n+\half} - \frac{2\tau}{\Delta t} \J_c^n = -c_0\tau \,\nablacp T_p^{n+\half} - c_h \tau \, \rotcp \ppsi_p^{n+\half},
			\end{equation*}
			or, equivalently,
			\begin{equation*}
				\J_c^{n+\half} = \frac{2\tau}{2\tau+\Delta t} \J_c^n-c_0\frac{\tau \, \Delta t}{2\tau+\Delta t} \,\nablacp T_p^{n+\half} - c_h\frac{\tau \Delta t}{2\tau + \Delta t} \, \rotcp \ppsi_p^{n+\half}.
			\end{equation*}
			Now, we replace $\tau = \kappa / c_0^2$ and apply a discrete divergence operator to the last equation, which cancels the curl term, and we further multiply by $c_0$ to obtain 
			\begin{equation*}
				c_0 \divpc\J_c^{n+\half} = \frac{2\kappa c_0}{2\kappa + \Delta tc_0^2} \divpc\J_c^n-\frac{\kappa c_0^2 \Delta t}{2\kappa + \Delta tc_0^2} \,\divpc\nablacp T_p^{n+\half}.
			\end{equation*}
			Therefore, for fixed values of $\Delta t$ and $\kappa$, one can expand in series the constant coefficients in the limit $c_0\to+\infty$ 
			\begin{equation}
				c_0 \divpc\J_c^{n+\half} =  \prn{\frac{2\kappa}{\Delta t c_0} + \bigO\prn{c_0^{-3}}} \divpc\J_c^n +\prn{-\kappa+\bigO\prn{c_0^{-2}}}  \divpc\nablacp T_p^{n+\half}.
				\label{eq:c0divJ}
			\end{equation}
			Now, we substitute the left-hand side using \eqref{eq:T2-disc} to obtain 
			\begin{equation}
				\frac{T_p^{n+1}-T_p^n}{\Delta t} = \prn{1+\bigO\prn{c_0^{-2}}} \prn{-\frac{2\kappa}{\Delta tc_0}\divpc\J_c^n + \kappa\,\divpc \nablacp T_p^{n+\half} }.
				\label{eq:discreteFourier1}
			\end{equation}
			Here we point out that the overall order of accuracy with respect to the discrete
			Fourier law depends on the scaling of $\divpc\J_c^n$. We show by induction that the latter only depends on the scaling of the initial datum. In fact, assume at the $n$-th iteration that 
			\begin{equation}
				\nrm{\mathbf T^{n}}_{L^2} = \bigO(1), \quad
				\nrm{\divpc\mathbf J^{n}}_{L^2} = \bigO(c_0^{-m}), \quad
				\text{where} \quad m > -1,
				\label{eq:assumptions_tn}
			\end{equation}
			and let us show that this scaling holds at the subsequent iteration. First, one can rewrite the discrete wave equation \eqref{eq:T2-eq} globally as
			\begin{equation*}
				\mathbb{A}_-\,\mathbf T^{n+1} = \mathbb{A}_+\,\mathbf T^{n}
				- \frac{2\beta(c_0)}{c_0}\,\divpc\mathbf J^{n},
				\qquad
				\beta(c_0)=\dfrac{\kappa}{1+\frac{2\kappa}{\Delta t}\frac{1}{c_0^2}}>0,
			\end{equation*}
			with $\mathbb{A}_\pm := \I \pm \tfrac{1}{2}\beta(c_0)\Delta t\,\mathbf{L}$. The negative semi-definiteness of $\mathbf{L}$ ensures that $\mathbb{A}_-$ is positive-definite. Moreover, since
			$\beta\le\kappa$, all the operators are bounded uniformly in $c_0$ in the discrete
			$L^2$ norm, so that
			\begin{equation*}
				\nrm{\mathbf{T}^{n+1}}_{L^2} \le C_1\,\nrm{\mathbf{T}^{n}}_{L^2}
				+ \frac{C_2}{c_0}\,\nrm{\divpc\mathbf J^{n}}_{L^2},
				\qquad C_1,C_2\ge0.
			\end{equation*}
			This inequality does not imply any maximum principle for $\mathbf{T}_p^n$, but ensures that
			any growth in time is independent of $c_0$, in particular
			$\nrm{\mathbf T^{n+1}}_{L^2}=\bigO(1)$.
			For $\divpc\mathbf J^{n+1}$, we recast \eqref{eq:c0divJ} using the induction hypothesis \eqref{eq:assumptions_tn} together with the bound just obtained, which gives
			\begin{align*}
				\nrm{\divpc\mathbf J^{n+1}}_{L^2}
				&\le \prn{1+\bigO\prn{c_0^{-2}}}\nrm{\divpc\mathbf J^{n}}_{L^2}
				+ \bigO\prn{c_0^{-1}}\,\nrm{\kappa\,\mathbf{L}\,\mathbf T^{n+\half}}_{L^2}
				\nonumber\\
				&= \bigO\prn{c_0^{-m}} + \bigO\prn{c_0^{-1}}
				= \bigO\prn{c_0^{-\min(m,1)}}.
			\end{align*}
			Inserting this into \eqref{eq:discreteFourier1} finally gives
			\begin{equation*}
				\nrm{\frac{\mathbf T^{n+1}-\mathbf T^n}{\Delta t}
					- \kappa\, \mathbf{L}\, \mathbf{T}^{n+\half}}_{L^2}
				= \bigO\prn{c_0^{-r}},
				\qquad r=\min\prn{2,\,1+m}.
			\end{equation*}
			Convergence requires in particular $r>0$, that is $m>-1$. 
		\end{proof}
		
		\subsubsection{Property 2: Discrete energy / discrete Lyapunov functional preservation}
		\begin{theorem}	\label{thm:discrete_energy} 	
			Under periodic boundary conditions, the global discrete total energy in the scheme~\eqref{eq:CattaneoGLM-disc} 
			\begin{equation*}
				\toten^n =  \half\left(\sum \limits_{p \in \Omega} |\Omega_p| \prn{T_p^n}^2 + 
				\sum \limits_{c \in \Omega} |\Omega_c| \prn{\nrm{\J_c^n}}^2 + 
				\sum \limits_{p \in \Omega} |\Omega_p| \prn{\nrm{\ppsi_p^n}}^2 + 
				\sum \limits_{c \in \Omega} |\Omega_c| \prn{\varphi_c^n}^2\right), 
			\end{equation*}	
			\begin{enumerate}
				\item[a)] is dissipated consistently with \eqref{eq:en_lyapunov}, \textit{i.e.}, 
				\begin{equation}
					\toten^{n+1} - \toten^n = - \Ediss <0,
					\label{eq:dissipated_en}
				\end{equation}
				where $\Ediss \coloneqq \dissipterm$, and
				\item[b)] is conserved exactly in the absence of sources. More precisely, in the limit $\tau\to\infty$, the dissipative scheme~\eqref{eq:CattaneoGLM-disc} reduces to the reversible scheme~\eqref{eq:heatGLM-disc}, and the discrete total energy is exactly conserved, \textit{i.e.}, 
				\begin{equation}
					\toten^{n} = \toten^0.
					\label{eq:conserved_en}
				\end{equation}
			\end{enumerate}
		\end{theorem}
		\begin{proof}
			We multiply each of the discrete equations \eqref{eq:CattaneoGLM-disc} by the corresponding conserved variable at time $t^{n+\half}$ and we integrate over their respective computational domains to obtain 
			\begin{subequations}
				\begin{align}
					\half\sum_{p\in\Omega}|\Omega_p|\prn{T_p^{n+1}}^2-\half\sum_{p\in\Omega}|\Omega_p|\prn{T_p^{n}}^2 
					=& -c_0\,\Delta t \sum_{p\in\Omega}|\Omega_p| T_p^{n+\half} \divpc\J_c^{n+\half} , 
					\label{eq:EN_1} \\
					\half\sum_{c\in\Omega}|\Omega_c|\prn{\J_c^{n+1}}^2
					-\half\sum_{c\in\Omega}|\Omega_c|\prn{\J_c^{n}}^2
					=& -c_0\,\Delta t \sum_{c\in\Omega}|\Omega_c|
					\,\J_c^{n+\half}\cdot\nablacp T_p^{n+\half} \nonumber \\
					 -c_h\,\Delta t \sum_{c\in\Omega}|\Omega_c| 
					\,\J_c^{n+\half}\cdot&\rotcp \ppsi_p^{n+\half}-\Delta t\sum_{c\in\Omega}|\Omega_c|
					\dfrac{|\J_c^{n+\half}|^2}{\tau},
					\label{eq:EN_2} \\
					\half\sum_{p\in\Omega}|\Omega_p|\prn{\ppsi_p^{n+1}}^2
					-\half\sum_{p\in\Omega}|\Omega_p|\prn{\ppsi_p^{n}}^2 
					=&\quad c_h\,\Delta t \sum_{p\in\Omega}|\Omega_p|
					\,\ppsi_p^{n+\half}\cdot\rotpc \J_c^{n+\half} \label{eq:EN_3}\\
					& -c_h\,\Delta t \sum_{p\in\Omega}|\Omega_p|
					\,\ppsi_p^{n+\half}\cdot\nablapc \varphi_c^{n+\half},\nonumber \\
					\half\sum_{c\in\Omega}|\Omega_c|\prn{\varphi_c^{n+1}}^2
					-\half\sum_{c\in\Omega}|\Omega_c|\prn{\varphi_c^{n}}^2
					=& -c_h\,\Delta t \sum_{c\in\Omega}|\Omega_c|
					\,\varphi_c^{n+\half}\,\divcp \ppsi_p^{n+\half}.
					\label{eq:EN_4}
				\end{align}
				\label{eq:discrete_dissipative}
			\end{subequations}
			By summing all these equations, the left-hand side terms simply amount to the total discrete energy difference $\toten^{n+1} - \toten^{n}$. On the right-hand side, the terms involving $\divpc\J_c^{n+\half}$ and $\nablacp T_p^{n+\half}$ from~\eqref{eq:EN_1} and~\eqref{eq:EN_2}, and those involving $\nablapc\varphi_c^{n+\half}$ and $\divcp\ppsi_p^{n+\half}$ from~\eqref{eq:EN_3} and~\eqref{eq:EN_4}, cancel pairwise by~\eqref{eq:Gauss_divgrad}. The cross terms involving $\rotcp\ppsi_p^{n+\half}$ and $\rotpc\J_c^{n+\half}$ from~\eqref{eq:EN_2} and~\eqref{eq:EN_3} cancel by~\eqref{eq:Gauss_rot}, leaving only the dissipation term and yielding~\eqref{eq:dissipated_en}.
		\end{proof}

		\subsubsection{Property 3: Discrete rotational energy conservation / dissipation}
		Define the discrete curls
		\begin{equation*}
			\wpJ = \rotpc \J_c, \qquad \wcpsi = \rotcp \ppsi_p.
		\end{equation*}
		\begin{theorem}
			\label{thm:discrete_curl}
			If periodic boundary conditions are imposed, the discrete rotational energy
			\begin{equation*}
				\toten_{\curl}^n = \half\sum_{p\in\Omega}|\Omega_p|  \nrm{\wpJn}^2 + \half\sum_{c\in\Omega} |\Omega_c| \nrm{\wcpsin}^2
			\end{equation*}
			in the scheme~\eqref{eq:CattaneoGLM-disc} satisfies the	following:
			\begin{enumerate}
				\item[a)] For the dissipative scheme~\eqref{eq:CattaneoGLM-disc}, the discrete rotational energy is dissipated, \textit{i.e.},
				\begin{equation}
					\toten_{\curl}^{n+1} - \toten_{\curl}^{n} = -\Edisscurl \leq 0, 
					\label{eq:curl_diss_disc}
				\end{equation}
				where $\Edisscurl \coloneqq \dissiptermcurl$.	
				\item[b)] In the limit $\tau \to \infty$, the discrete rotational energy is conserved, \textit{i.e.},
				\begin{equation*}
					\toten_{\curl}^{n} = \toten_{\curl}^{0}.
				\end{equation*}
			\end{enumerate}
		\end{theorem}	
		\begin{proof}
			\begin{enumerate}
				\item[a)] Applying $\rotpc$ to \eqref{eq:J2-disc} and $\rotcp$ to \eqref{eq:psi2-disc}, multiplying the resulting equations by $|\Omega_p|\,\wpJnph$ and $|\Omega_c|\,\wcpsinph$ respectively, and summing over all cells, one obtains
				\begin{subequations}
					\begin{align}
						& \half\sum_{p\in\Omega}|\Omega_p|\nrm{\wpJnp}^2 - \half\sum_{p\in\Omega}|\Omega_p|\nrm{\wpJn}^2 \nonumber 
						\\
						\quad = & - c_h\,\Delta t\sum_{p\in\Omega}|\Omega_p| \,\wpJnph \cdot \rotpc\wcpsinph
						- \frac{\Delta t}{\tau}\sum_{p\in\Omega}|\Omega_p| \nrm{\wpJnph}^2,
						\label{eq:curl_1} \nonumber \\
						& \half\sum_{c\in\Omega}|\Omega_c|\nrm{\wcpsinp}^2 - \half\sum_{c\in\Omega}|\Omega_c|\nrm{\wcpsin}^2 \\
						\quad = & \phantom{-}  c_h\,\Delta t\sum_{c\in\Omega}|\Omega_c|\, \wcpsinph \cdot \rotcp\wpJnph.
						\label{eq:curl_2}
					\end{align}
				\end{subequations}
				Summing~\eqref{eq:curl_1} and~\eqref{eq:curl_2}, the cross terms involving $\rotpc\wcpsinph$ and $\rotcp\wpJnph$ cancel by~\eqref{eq:Gauss_rot}, leaving only the dissipation term and yielding~\eqref{eq:curl_diss_disc}.	
				\item[b)] In the limit $\tau \to \infty$, the dissipation term in~\eqref{eq:curl_1} vanishes, and~\eqref{eq:curl_diss_disc} reduces to $\toten_{\curl}^{n+1}~=~\toten_{\curl}^{n}$, so that $\toten_{\curl}^n = \toten_{\curl}^0$ for all $n \geq 0$.
			\end{enumerate}
		\end{proof}	
		Therefore, if $\rotpc \J_c^0 = 0$ and $\rotcp \ppsi_p^0 = 0$, then $\rotpc \J_c^n = 0$ and  $\rotcp \ppsi_p^n = 0$, $\forall n\ge 0.$
		
		%
		
		\subsubsection{Property 4: Discrete grad-div energy and discrete wave invariant conservation}
		
		Here we show that a discrete analogue of Property \ref{prop:wave_inv} holds. We first consider the grad-div energy conservation.  
		\begin{theorem}
			\label{thm:discrete_grad_div}
			The discrete grad-div energy
			\begin{equation*}
				\toten_{T}^n =\half\sum_{c\in\Omega} |\Omega_c|\nrm{\nablacp T_p^n}^2 +\half\sum_{p\in\Omega} |\Omega_p|\nrm{\divpc \J_c^n}^2
			\end{equation*}
			in the scheme~\eqref{eq:CattaneoGLM-disc} satisfies the following:
			\begin{enumerate}
				\item[a)] For the dissipative scheme~\eqref{eq:CattaneoGLM-disc}, the discrete grad-div energy is dissipated, \textit{i.e.},
				\begin{equation}
					\toten_{T}^{n+1}-\toten_{T}^{n}	= -\Edissdiv\leq 0,
					\label{eq:graddiv_disc_diss}
				\end{equation}
				where $\Edissdiv \coloneqq \dissiptermdiv$.	
				\item[b)] In the limit $\tau\to\infty$, the discrete grad-div energy is exactly conserved, \textit{i.e.},
				\begin{equation*}
					\toten_{T}^{n}=\toten_{T}^{0}.
				\end{equation*}
			\end{enumerate}
		\end{theorem}
		\begin{proof}
			\begin{enumerate}
				\item[a)] Applying $\nablacp$ to~\eqref{eq:T2-disc} and $\divpc$ to~\eqref{eq:J2-disc}, multiplying the resulting 
				equations by $|\Omega_c|\,\nablacp T_p^{n+\half}$ and $|\Omega_p|\,\divpc \J_c^{n+\half}$ respectively, and summing over all cells, one obtains
				\begin{subequations}
					\begin{align}
						& \half\sum_{c\in\Omega}|\Omega_c|\nrm{\nablacp T_p^{n+1}}^2 - \half\sum_{c\in\Omega}|\Omega_c| \nrm{\nablacp T_p^{n}}^2 \nonumber \\ 
						\quad &	= -c_0\,\Delta t\sum_{c\in\Omega}|\Omega_c|\,	\nablacp T_p^{n+\half}\cdot	\nablacp\prn{\divpc \J_c^{n+\half}},
						\label{eq:gd_1} \\
						&\half\sum_{p\in\Omega}|\Omega_p|\nrm{\divpc \J_c^{n+1}}^2- \half\sum_{p\in\Omega}|\Omega_p|	\nrm{\divpc \J_c^{n}}^2 \nonumber\\ 
						\quad & = -c_0\,\Delta t\sum_{p\in\Omega}|\Omega_p|\,\divpc \J_c^{n+\half}\divpc\prn{\nablacp T_p^{n+\half}}
						- \frac{\Delta t}{\tau}\sum_{p\in\Omega}|\Omega_p|\nrm{\divpc \J_c^{n+\half}}^2.
						\label{eq:gd_2}
					\end{align}
				\end{subequations}	
				Summing~\eqref{eq:gd_1} and~\eqref{eq:gd_2}, the cross terms involving $\nablacp T_p^{n+\half}$ and $\divpc \J_c^{n+\half}$ cancel by~\eqref{eq:Gauss_divgrad}, leaving only the dissipation term and yielding~\eqref{eq:graddiv_disc_diss}.
				\item[b)] In the limit $\tau\to\infty$, the dissipation term in~\eqref{eq:gd_2} vanishes and~\eqref{eq:graddiv_disc_diss} reduces to $\toten_T^{n+1} = \toten_T^{n}$, so that $\toten_T^n = \toten_T^0$ for all $n \geq 0$.
			\end{enumerate}
		\end{proof}
		As a consequence, if $\nablacp T_p^0 = 0$ and $\divpc \J_c^0 = 0$ then, in the reversible case, these involutions are preserved for all times, that is, $\nablacp T_p^n = 0$ and $\divpc \J_c^n = 0$ for all $n\ge 0.$
		
		Now we consider the discrete analogue of the second part of Property \ref{prop:wave_inv}.
		\begin{theorem}
			\label{thm:discrete_wave_inv}
			If periodic boundary conditions are imposed, the discrete wave invariant
			\begin{equation*}
				\toten_{\varphi}^n = \half\sum_{p\in\Omega}|\Omega_p|\nrm{\nablapc \varphi_c^n}^2 + \half\sum_{c\in\Omega}|\Omega_c|\nrm{\divcp \ppsi_p^n}^2
			\end{equation*}
			is exactly conserved by both the reversible scheme~\eqref{eq:heatGLM-disc} and the 
			dissipative scheme~\eqref{eq:CattaneoGLM-disc}, \textit{i.e.},
			\begin{equation*}
				\toten_{\varphi}^{n} = \toten_{\varphi}^{0}, \qquad \forall n \geq 0.
			\end{equation*}
		\end{theorem}
		
		\begin{proof}
			Applying $\nablapc$ to~\eqref{eq:phi2-disc} and $\divcp$ to~\eqref{eq:psi2-disc}, multiplying the resulting equations by $|\Omega_p|\,\nablapc\varphi_c^{n+\half}$ and $|\Omega_c|\,\divcp\ppsi_p^{n+\half}$ respectively, and summing over all cells, one obtains
			\begin{subequations}
				\begin{align}
					& \half\sum_{p\in\Omega}|\Omega_p| \nrm{\nablapc\varphi_c^{n+1}}^2 - \half\sum_{p\in\Omega}|\Omega_p| \nrm{\nablapc\varphi_c^{n}}^2
					\nonumber \\
					\quad =  &   -c_h\,\Delta t\sum_{p\in\Omega}|\Omega_p|\, \nablapc\varphi_c^{n+\half}\cdot
					\nablapc\prn{\divcp\ppsi_p^{n+\half}},
					\label{eq:wi_1} \\
					& \half\sum_{c\in\Omega}|\Omega_c| \nrm{\divcp\ppsi_p^{n+1}}^2 - \half\sum_{c\in\Omega}|\Omega_c| \nrm{\divcp\ppsi_p^{n}}^2 \nonumber \\
					\quad = & -c_h\,\Delta t\sum_{c\in\Omega}|\Omega_c| \,\divcp\ppsi_p^{n+\half} \divcp\prn{\nablapc\varphi_c^{n+\half}}.
					\label{eq:wi_2}
				\end{align}
			\end{subequations}
			Summing~\eqref{eq:wi_1} and~\eqref{eq:wi_2}, the cross terms involving $\nablapc\varphi_c^{n+\half}$ and $\divcp\ppsi_p^{n+\half}$ cancel by~\eqref{eq:Gauss_divgrad}, and since no dissipation term is present, this yields 
			$\toten_\varphi^{n+1} = \toten_\varphi^n$, so that $\toten_\varphi^n = \toten_\varphi^0$ for all $n \geq 0$.
		\end{proof}
		As a consequence, if $\nablapc\varphi_c^0 = 0$ and $\divcp\ppsi_p^0 = 0$, then these involutions are preserved for all times, that is, $\nablapc\varphi_c^n = 0$ and $\divcp\ppsi_p^n = 0$ for all $n \geq 0$.
		
		The five theorems above constitute the discrete versions of the continuous results proven in Sections~\ref{sec:gov_eq} and~\ref{sec:cont_constraints}. Theorem~\ref{thm:AP} guarantees that the scheme captures the correct parabolic limit without any restriction on the time step. Theorems~\ref{thm:discrete_energy}, \ref{thm:discrete_curl}, and \ref{thm:discrete_grad_div} show that the discrete total energy, rotational energy, and grad-div energy mirror exactly the behavior of their continuous counterparts: they are dissipated at a consistent rate in the presence of relaxation, and conserved exactly in its absence. Theorem~\ref{thm:discrete_wave_inv} shows that the discrete wave invariant $\toten_\varphi^n$ is exactly conserved regardless of whether relaxation is present, showing that $\varphi$ and $\ppsi$ are entirely decoupled from the dissipation mechanism at both the continuous and discrete 
		levels. In all cases, the curl and grad-div involutions are preserved at the discrete level independently of the mesh size and time step, which is a direct consequence of the compatibility 
		of the discrete operators defined in Section~\ref{sec:num_discretization}.
		
		\section{Numerical results}\label{sec:num_results}
		\subsection{Numerical convergence analysis for the reversible system}\label{sec:num_results_rev}
		In order to show that the proposed scheme is second-order accurate, we consider a
		\textit{planar wave} solution, propagating along the direction $\bn = (\cos\theta, \sin\theta)$,
		with $\theta = -\pi/4$. The initial conditions read
		\begin{gather*}
			T(x,y,0) = 0.25\sin \left(\pi (x-y)\right), \qquad \J(x,y,0)  = \J_0 \sin\left( \pi (x-y) \right),	\\
			\ppsi(x,y,0) = \ppsi_0 \sin\left( \pi (x-y) \right), \qquad \varphi(x,y,0) = 0.5  \sin\left(\pi (x-y)\right), \\ 
			\text{where} \ \ 
			\J_0 = \mathbf{R}(\theta)\cdot[0.25,\,0,\,1]\T = [0.25b,\,-0.25b,\,1]\T, \quad
			\ppsi_0 = \mathbf {R}(\theta)\cdot[0.5,\,1,\,0]\T = [1.5b,\,0.5b,\,0]\T,
		\end{gather*}
		$b=\sqrt2/2$ and $\mathbf{R}(\theta)$ is the rotation matrix of angle $\theta$. The computational domain is $\Omega=[-1,1]^2$ with periodic boundary conditions. The relaxation source term is set to zero here, and the reversible scheme~\eqref{eq:heatGLM-disc} is used throughout. The wave speeds are set to $c_0=1$ and $c_h=1$, so that the solution returns to its initial state at $t = \lambda = \sqrt{2}$, at which time the obtained numerical solution is compared with the initial condition. Since the scheme is unconditionally stable, the time step is chosen purely on accuracy grounds, with a CFL number of $0.9$ and unit wave speed. We consider a sequence of uniform Cartesian meshes with $N\times N$ computational cells, with $N\in\{ 16, 32, 64, 128\}$.  
		\begin{figure}[htb]
			\centering
			\includegraphics[trim=150 250 150 250,clip,width=0.4\textwidth]{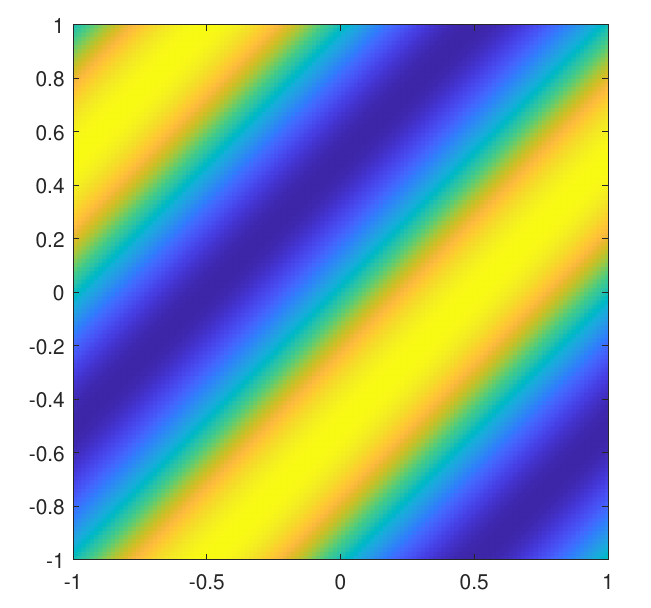}\quad
			\includegraphics[width=0.425\textwidth]{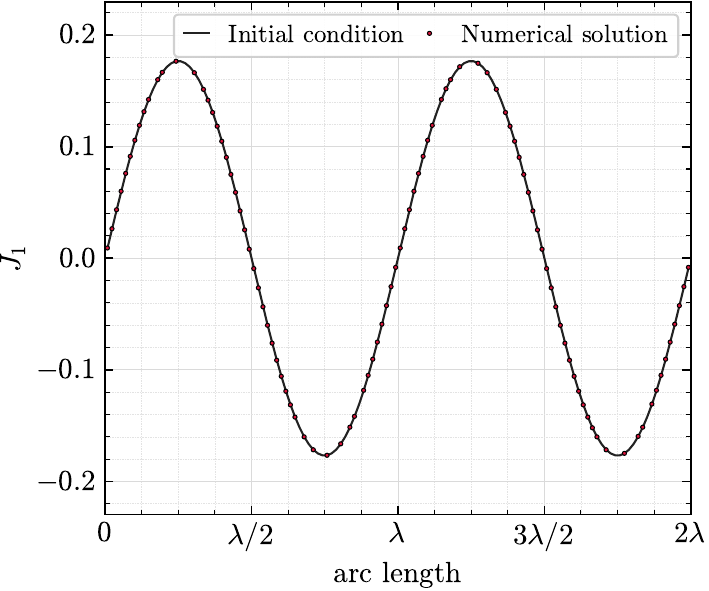}
			\caption{Left: approximate solution at $t=\sqrt{2}$ for $J_1$, reversible system. Right: comparison between the initial condition and the approximate solution along the anti-diagonal cut $y=-x$.}
			\label{fig:solution_reversible}
		\end{figure}
		Figure~\ref{fig:solution_reversible} shows the first component of the heat flux $J_1$ at the final time (left), and a pointwise comparison between the initial condition and the numerical solution at final time (right) on a 1D cut through the domain along the $x$ axis, which shows perfect agreement. The $L^2$ errors for the relevant variables are reported in Table \ref{tab:SIMM_heat_err}. The results clearly show second-order convergence for all the variables.  
		\begin{table}[htb]
			\centering
			\renewcommand{\arraystretch}{1.1}
			\begin{tabular}{|l|cccc|ccc|}
				\hline
				&\multicolumn{4}{c|}{$L^2$ errors} & \multicolumn{3}{c|}{Convergence order} \\				
				$N$	&16		&32		&64		&128 & & & \\
				\hline	
				$T$	        &6.71$\times10^{-2}$	&1.71$\times10^{-2}$	&4.28$\times10^{-3}$	&1.07$\times10^{-3}$	&1.98	&1.99	&2.00	\\	
				$J_1$	    &4.74$\times10^{-2}$	&1.21$\times10^{-2}$	&3.03$\times10^{-3}$	&7.58$\times10^{-4}$	&1.98	&1.99	&2.00	\\
				$J_2$	    &4.74$\times10^{-2}$	&1.21$\times10^{-2}$	&3.03$\times10^{-3}$	&7.58$\times10^{-4}$	&1.98	&1.99	&2.00	\\
				$J_3$	    &2.68$\times10^{-1}$	&6.82$\times10^{-2}$	&1.71$\times10^{-2}$	&4.29$\times10^{-3}$	&1.98	&1.99	&2.00	\\
				$\psi_1$	&2.85$\times10^{-1}$	&7.24$\times10^{-2}$	&1.82$\times10^{-2}$	&4.55$\times10^{-3}$	&1.98	&1.99	&2.00	\\
				$\psi_2$	&9.49$\times10^{-2}$	&2.41$\times10^{-2}$	&6.06$\times10^{-3}$	&1.52$\times10^{-3}$	&1.98	&1.99	&2.00	\\
				$\varphi$	    &1.34$\times10^{-1}$	&3.41$\times10^{-2}$	&8.57$\times10^{-3}$	&2.14$\times10^{-3}$	&1.98	&1.99	&2.00	\\
				\hline
			\end{tabular}
			\caption{$L^2$ error norms and corresponding convergence orders for the planar wave traveling in the direction $\bn = (1,-1)$, obtained with the fully discrete semi-implicit scheme on uniform grids composed of $N \times N$ elements. }
			\label{tab:SIMM_heat_err}
		\end{table}
		
		Figure~\ref{fig:InvEn_heat} shows the temporal evolution of the errors in the four discrete conserved quantities: the total energy $\toten^n$, the rotational energy $\toten_{\curl}^n$, and the two wave invariants $\toten_T^n$ and $\toten_\varphi^n$..

		\begin{figure}[htb]
			\centering
			\begin{subfigure}[b]{0.48\textwidth}
				\centering
				\includegraphics[trim=0 0 0 0,clip,width=1.10\textwidth]{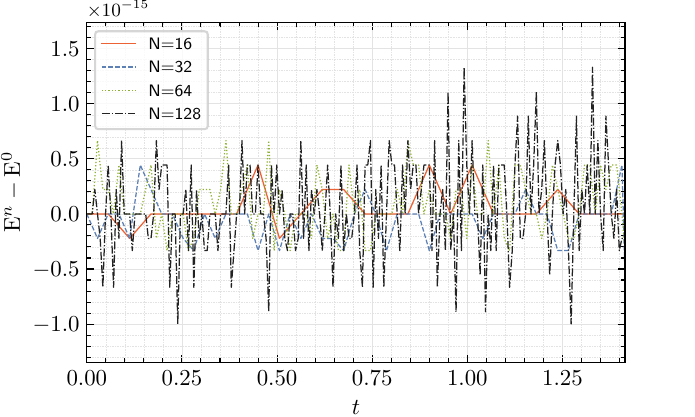}
				\label{fig:en_heat}
			\end{subfigure}
			\hfill
			\begin{subfigure}[b]{0.48\textwidth}
				\centering
				\includegraphics[trim=0 0 0 0,clip,width=1.\textwidth]{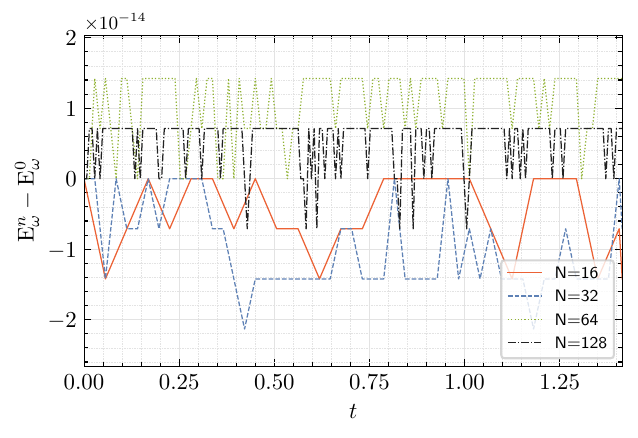}
				\label{fig:curl_en_heat}
			\end{subfigure}
			\\[-1em]
			\begin{subfigure}[b]{0.48\textwidth}
				\centering
				\includegraphics[trim=0 0 0 0,clip,width=\textwidth]{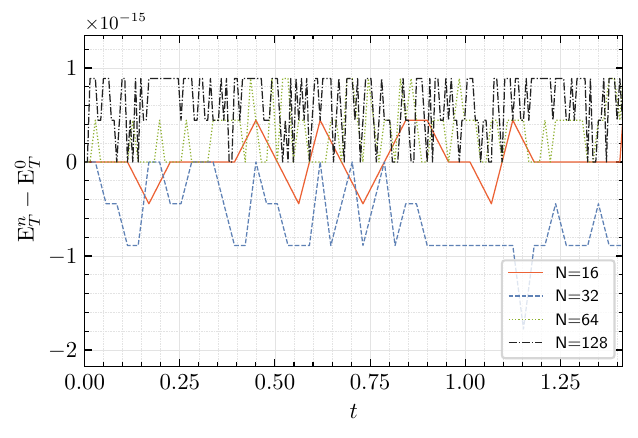}
				\label{fig:graddiv_en_heat}
			\end{subfigure}
			\hfill
			\begin{subfigure}[b]{0.48\textwidth}
				\centering
				\includegraphics[trim=0 0 0 0,clip,width=1.\textwidth]{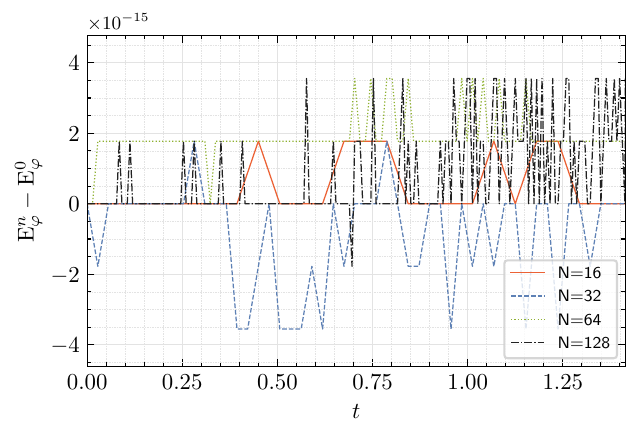}
				\label{fig:graddiv2_en_heat}
			\end{subfigure}
			\caption{Temporal evolution of the discrete conservation errors (top left: total energy $\toten^n - \toten^0$, top right: rotational energy $\toten_{\curl}^n - \toten_{\curl}^0$, bottom left: grad-div energy $\toten_T^n - \toten_T^0$, bottom right: wave invariant $\toten_\varphi^n - \toten_\varphi^0$) for the reversible system, on successively refined uniform $N\times N$ grids. All four quantities are preserved up to machine precision for all resolutions, confirming the discrete conservation results of Theorems~\ref{thm:discrete_energy},~\ref{thm:discrete_curl},~\ref{thm:discrete_grad_div}, and~\ref{thm:discrete_wave_inv}.}
			\label{fig:InvEn_heat}
		\end{figure}
		
		\subsection{Numerical convergence analysis for the irreversible system}
		Here we conduct a test analogous to the one above, but in the presence of relaxation source terms. We consider a one-dimensional exact solution of the equations, corresponding to a standing decaying mode of the heat subsystem (see Appendix \ref{app:exact_sol} for details), rotated with $\theta=-\pi/4$. The initial data are prescribed as follows
		\begin{gather*}
			T(x,y,0) = \sin\!\left(\pi (x-y) \right), \qquad \J(x,y,0) = \J_0 \cos\!\left( \pi (x-y) \right) , \\
			\ppsi(x,y,0) = \ppsi_0 \sin\!\left( \pi(x-y) \right), \qquad \varphi(x,y,0) = 2\sin\!\left( \pi(x-y) \right), \\ 
			\J_0 = \mathbf{R}(\theta)\cdot \left[\frac{-\alpha}{\pi \sqrt2 c_0},0,0\right]\T = \left[\frac{-\alpha}{2\pi c_0},\frac{\alpha}{2\pi c_0},0\right]\T,  \qquad \ppsi_0 = \mathbf{R}(\theta)\cdot \left[2,0,0\right]\T = \left[\sqrt2,-\sqrt2,0\right]\T,
		\end{gather*}
		where $\alpha = 4\kappa\pi^2\prn{1 + \sqrt{1-8\pi^2\kappa^2/c_0^{2}}}^{-1}$. The exact solution in this case is given by 
		\begin{alignat*}{2}
			&T(x,y,t) = \sin\!\left(\pi (x-y) \right)e^{-\alpha t}, \qquad 	&&\J(x,y,t) = \J_0 \cos\!\left( \pi (x-y) \right)e^{-\alpha t} \\
			&\ppsi(x,y,t) = \ppsi_0 \sin\!\left(\pi (x-y-\sqrt{2}\,c_h t) \right), \qquad 	&&\varphi(x,y,t) = 2\sin\!\left( \pi (x-y-\sqrt{2}\,c_h t) \right).
		\end{alignat*}
		The computational domain is $\Omega=[-1,1]^2$, with periodic boundary conditions. We take $c_0=10$, $c_h=1$, and $\tau=\kappa/c_0^2$ with $\kappa=10^{-2}$, so that $\alpha\in\RR^+$. The time step is again chosen with a CFL number of $0.9$ and unit wave speed, independent of the system parameters. The final time is $t_{\mathrm{end}}=\sqrt{2}$, corresponding to one period of the planar wave in the cleaning variables. 
		\begin{figure}[!h]
			\centering
			\includegraphics[trim=150 250 150 250,clip,width=0.4\textwidth]{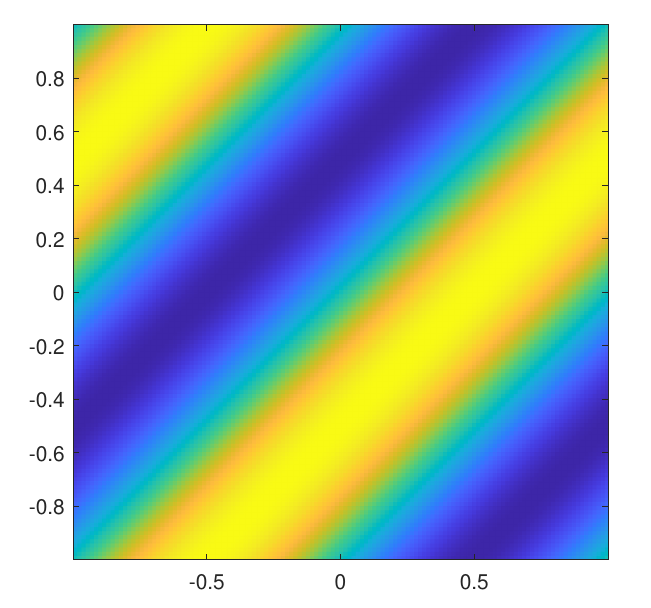}\quad
			\includegraphics[width=0.43\textwidth]{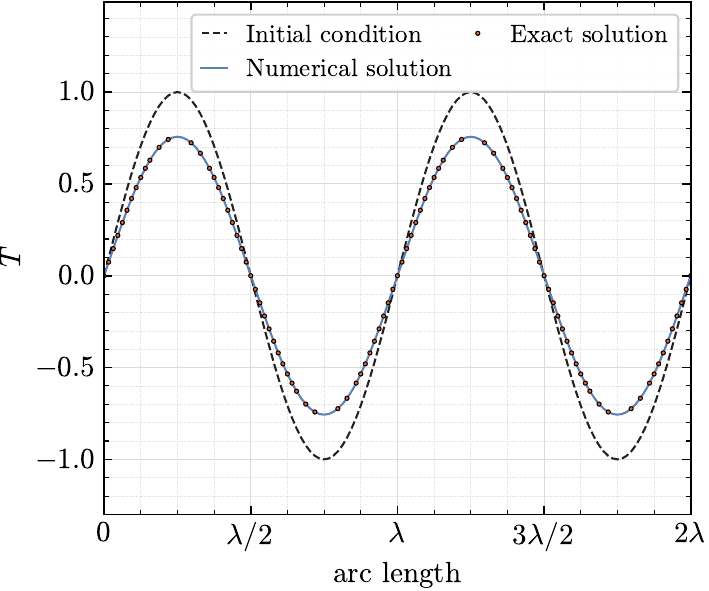}
			\caption{Left: numerical solution at $t=\sqrt{2}$ for $T$, dissipative system. Right: pointwise comparison between the exact and the approximate solutions along the anti-diagonal cut $y=-x$ at $t=\sqrt{2}$.}
			\label{fig:solution_dissipative}
		\end{figure}
		Figure~\ref{fig:solution_dissipative} shows the temperature field at the final time (left), and a pointwise comparison between the exact and the approximate solutions along the cut $y=-x$ (right). The two curves are in excellent agreement, confirming the scheme's accuracy in the dissipative case. A convergence study is performed on a sequence of successively refined Cartesian meshes with $N\in\{16,32,64,128\}$. The corresponding $L^2$ errors are shown in Table~\ref{tab:SIMM_Cattaneo_err} and confirm that the proposed scheme for the relaxed system retains second-order accuracy. 
		
		\begin{table}[htb]
			\centering
			\renewcommand{\arraystretch}{1.1}
			\begin{tabular}{|l|cccc|ccc|}
				\hline
				&\multicolumn{4}{c|}{$L^2$ errors} & \multicolumn{3}{c|}{Convergence order} \\				
				$N$	&16	&32	&64	&128 & & & \\
				\hline		
				$T$	    &1.51$\times10^{-2}$ &3.82$\times10^{-3}$ &9.58$\times10^{-4}$ &2.40$\times10^{-4}$ 
				&1.99 &2.00 &2.00\\
				$J_1$	&5.09$\times10^{-5}$ &2.29$\times10^{-5}$ &2.78$\times10^{-6}$ &5.97$\times10^{-7}$ 
				&1.15 &3.04 &2.22 \\
				$\psi_1$&3.80$\times10^{-1}$ &9.65$\times10^{-2}$ &2.42$\times10^{-2}$ &6.07$\times10^{-3}$ 
				&1.98 &1.99 &2.00 \\
				$\varphi$	&5.37$\times10^{-1}$ &1.36$\times10^{-1}$ &3.43$\times10^{-2}$ &8.58$\times10^{-3}$ 
				&1.98 &1.99 &2.00 \\
				\hline	
			\end{tabular}
			\caption{$L^2$ error norms and corresponding convergence orders for the decaying planar wave test case, obtained with the fully discrete semi-implicit scheme on uniform grids composed of $N \times N$  elements.}
			\label{tab:SIMM_Cattaneo_err}
		\end{table}

		\begin{figure}[!h]
			\centering
			\begin{subfigure}[b]{0.48\textwidth}
				\centering
				\includegraphics[trim=0 0 0 0,clip,width=1.0\textwidth]{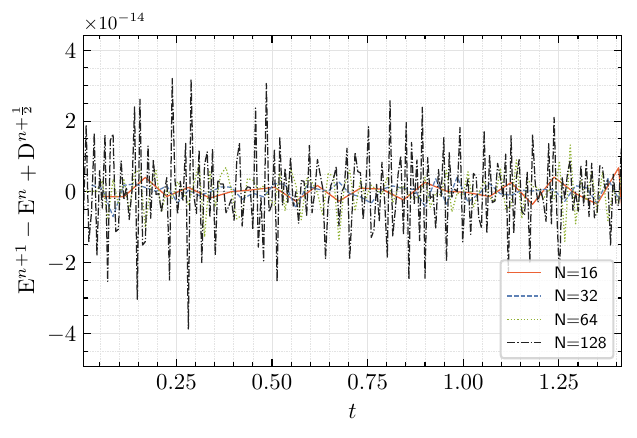}
				\label{fig:en_Cattaneo}
			\end{subfigure}
			\hfill
			\begin{subfigure}[b]{0.48\textwidth}
				\centering
				\includegraphics[trim=0 0 0 0,clip,width=1.\textwidth]{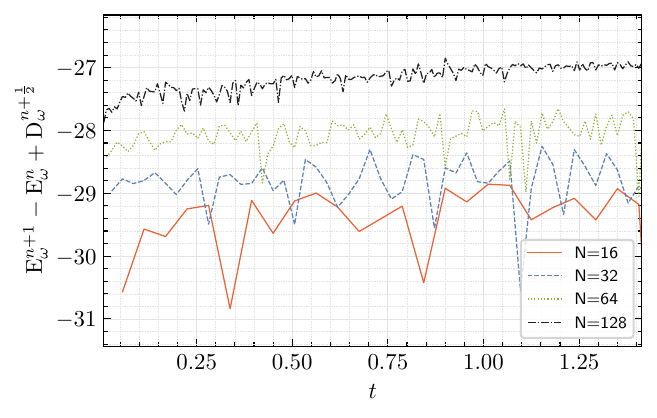}
				\label{fig:curl_en_Cattaneo}
			\end{subfigure}
			\\[-1em]
			\begin{subfigure}[b]{0.48\textwidth}
				\centering
				\includegraphics[trim=0 0 0 0,clip,width=\textwidth]{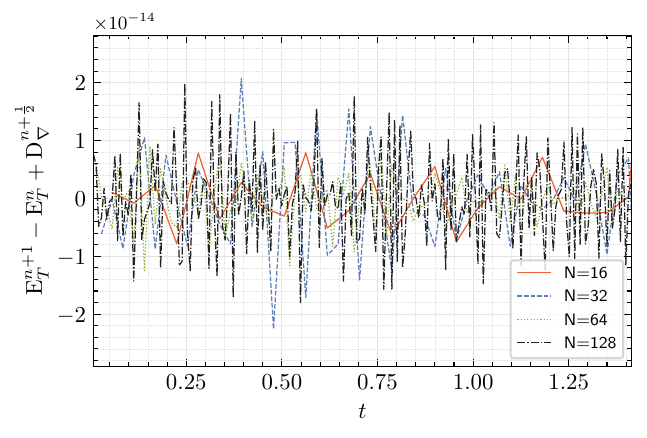}
				\label{fig:graddiv_en_Cattaneo}
			\end{subfigure}
			\hfill
			\begin{subfigure}[b]{0.48\textwidth}
				\centering
				\includegraphics[trim=0 0 0 0,clip,width=1.\textwidth]{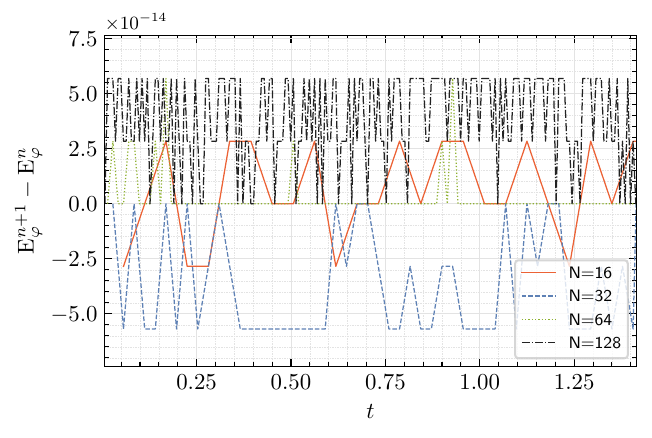}
				\label{fig:graddiv2_en_Cattaneo}
			\end{subfigure}
			\caption{Temporal evolution of the discrete dissipation errors (top left: total energy $\toten^{n+1} - \toten^{n} + \Ediss$, top right: rotational energy $\toten_{\curl}^{n+1}- \toten_{\curl}^{n}+\Edisscurl$ in logarithmic scale, bottom left: grad-div energy $\toten_T^{n+1} - \toten_T^{n}+\Edissdiv$, bottom right: wave invariant $\toten_\varphi^{n+1} - \toten_\varphi^{n}$) for the dissipative system, on successively refined uniform $N\times N$ grids. The four quantities are preserved up to machine precision for all resolutions, confirming the discrete conservation results of Theorems~\ref{thm:discrete_energy},~\ref{thm:discrete_curl},~\ref{thm:discrete_grad_div}, and~\ref{thm:discrete_wave_inv}.}
			\label{fig:InvEn_Cattaneo}
		\end{figure}
		
		Figure~\ref{fig:InvEn_Cattaneo} shows the temporal evolution of the errors in the four discrete quantities: the total energy $\toten^n$, the rotational energy $\toten_{\curl}^n$, and the two wave invariants $\toten_T^n$ and $\toten_\varphi^n$. As proven in Theorems~\ref{thm:discrete_energy}--\ref{thm:discrete_wave_inv}, the quantities $\toten^n$, $\toten_{\curl}^n$ and $\toten_T^n$ decay with dissipation rates $\Ediss$, $\Edisscurl$, and $\Edissdiv$, respectively, while $\toten_\varphi^n$ is preserved up to machine precision.

		\subsection{Asymptotic preservation of the Fourier limit}
		Next, we will numerically verify the asymptotic-preserving property established in Theorem~\ref{thm:AP}. The theorem predicts that, for initial data satisfying $\divpc \J_c^0=\mathcal O(c_0^{-m}),$ the discrete solution converges towards the Fourier limit with asymptotic rate $r=\min(2,1+m)$. 
		
		For this problem, the computational domain is set to $\Omega=[-1,1]^2$, periodic boundary conditions are imposed, and the final time is set to $t_{\mathrm{end}}=1$. The mesh is fixed to $50\times 50$ cells, $c_h=1$, $\kappa=10^{-2}$, while the characteristic speed $c_0$ is progressively increased from $10^2$ to $10^5$. The initial data are given by a Gaussian profile, $G(x,y)=\exp \left( -\half \nrm{\x}^2/\sigma^2 \right),$ with $\sigma=0.2$. More specifically, we set 
		\begin{equation*}
			\varphi(x,y,0)=G(x,y), \qquad
			\ppsi(x,y,0)=\left(G(x,y),G(x,y),G(x,y)\right)\T, \qquad T(x,y,0)=G(x,y).
		\end{equation*}
		For the field $\J$, we consider an initial field with a scalar prefactor $A=A(c_0)$, such that
		\begin{equation*}
			\J(x,y,0)= A\left( 
			\sin\!\left(-\dfrac{ \nrm{\x}^2}{2\sigma^2}\right),
			\cos\!\left(-\dfrac{ \nrm{\x}^2}{2\sigma^2}\right),
			0 \right)\T,
			\mbox{ with } A\in \left\{ c_0^{1/2},\, 1,\, c_0^{-1/2},\, c_0^{-1},\, c_0^{-2},\, 0 \right\}.
		\end{equation*}
		Theorem~\ref{thm:AP} predicts asymptotic convergence rates $r=\frac12,1,\frac32,2,2,2$. Additionally, we consider a well-prepared initial condition satisfying the discrete Fourier equilibrium relation
		\begin{equation}
			\J(x,y,0)=-\sqrt{\kappa\tau}\,\nabla T(x,y,0),
			\label{eq:well-prepared-J}
		\end{equation}
		for which second-order convergence is again expected, but with an improved error constant, at least at the continuous level, by virtue of Lemma~\ref{lmm:Fourier}. For each value of $c_0$, the $L^2$ norm of the residual \eqref{eq:discrete_residual} is measured and reported alongside the convergence rates in Table~\ref{tab:fourier_limit}. The numbers are in agreement with the prediction of Theorem~\ref{thm:AP}. In particular, we recover the correct convergence orders, even for initial data whose discrete divergence grows with $c_0$, that is, for $m=-1/2$, and saturation is reached at second order.  
		\begin{table}[htbp]
			\centering
			\small\setlength{\tabcolsep}{4pt}
			\begin{tabular}{l|cccc|ccc}
				\hline
				& $c_0=10^2$ & $c_0=10^3$ & $c_0=10^4$ & $c_0=10^5$
				& \multicolumn{3}{c}{Order} \\
				\hline
				$A=c_0^{1/2}$ 
				& $2.886{\times}10^{0}$ & $9.238{\times}10^{-1}$ & $2.922{\times}10^{-1}$ & $9.239{\times}10^{-2}$
				& $0.49$ & $0.50$ & $0.50$ \\
				
				$A=1$
				& $2.886{\times}10^{-1}$ & $2.921{\times}10^{-2}$ & $2.922{\times}10^{-3}$ & $2.922{\times}10^{-4}$
				& $0.99$ & $1.00$ & $1.00$ \\
				
				$A=c_0^{-1/2}$
				& $2.886{\times}10^{-2}$ & $9.238{\times}10^{-4}$ & $2.922{\times}10^{-5}$ & $9.239{\times}10^{-7}$
				& $1.49$ & $1.50$ & $1.50$ \\
				
				$A=c_0^{-1}$
				& $2.886{\times}10^{-3}$ & $2.921{\times}10^{-5}$ & $2.922{\times}10^{-7}$ & $2.922{\times}10^{-9}$
				& $1.99$ & $2.00$ & $2.00$ \\
				
				$A=c_0^{-2}$
				& $3.187{\times}10^{-5}$ & $1.399{\times}10^{-7}$ & $1.368{\times}10^{-9}$ & $1.368{\times}10^{-11}$
				& $2.36$ & $2.01$ & $2.00$ \\
				
				$A=0$
				& $1.352{\times}10^{-5}$ & $1.368{\times}10^{-7}$ & $1.368{\times}10^{-9}$ & $1.368{\times}10^{-11}$
				& $1.99$ & $2.00$ & $2.00$ \\
				
				$J^0=-\sqrt{\kappa\tau}\,\nabla T^0$
				& $3.933{\times}10^{-8}$ & $3.934{\times}10^{-10}$ & $3.934{\times}10^{-12}$ & $4.017{\times}10^{-14}$
				& $2.00$ & $2.00$ & $1.99$ \\
				\hline
			\end{tabular}
			\caption{Asymptotic Fourier limit test. Errors at $t_{\mathrm{end}}=1$ for increasing values of $c_0$, with orders computed for different values of $c_0$.}
			\label{tab:fourier_limit}
		\end{table}
		
		The results of the asymptotic-preserving study are summarized in Figure~\ref{fig:Fourier}. The log--log representation of the $L^2$ error as a function of $c_0$ shows a decay whose slope depends on the scaling of the initial data. For the initial conditions characterized by $A=\mathcal{O}(c_0^{-m})$, the observed asymptotic rates match perfectly the theoretical estimates given by Theorem~\ref{thm:AP}, that is, $r=\min(2,1+m)$. In particular, the cases $A=c_0^{1/2}$, $A=1$, $A=c_0^{-1/2}$ and $A=c_0^{-1}$ exhibit convergence rates equal to $1/2$, $1$, $3/2$ and $2$, respectively. As predicted by the analysis, the convergence is second-order for all well-prepared initial data, that is, whenever $\divpc \J_c^0=\mathcal{O}(c_0^{-1})$ or smaller. This behavior is observed for the cases $A=c_0^{-1}$, $A=c_0^{-2}$, and $A=0$, whose error curves show the same slope. Finally, the initial condition satisfying the discrete Fourier equilibrium relation~\eqref{eq:well-prepared-J} produces errors that are several orders of magnitude smaller and approach machine precision for the largest values of $c_0$. 
		
		\begin{figure}[!h]
			\centering
			\includegraphics[trim=0 0 0 0,clip,width=0.7\textwidth]{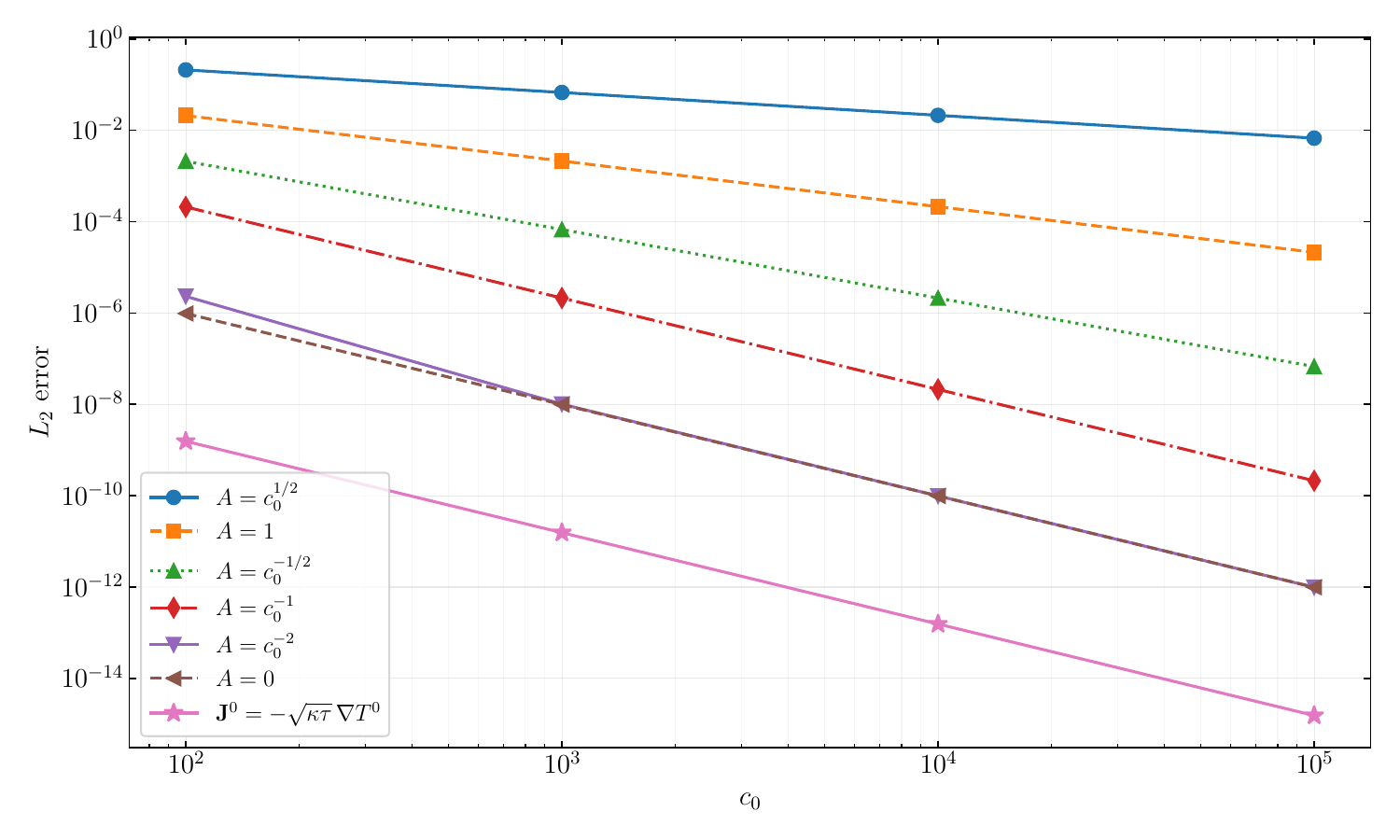}
			\caption{Verification of the asymptotic-preserving property. $L^2$ error with respect to the Fourier limit solution as a function of $c_0$ for different scalings of the initial condition. The observed slopes agree with the theoretical prediction $r=\min(2,1+m)$ of Theorem~\ref{thm:AP}.}
			\label{fig:Fourier}
		\end{figure}
		
		\section{Conclusion and perspectives}
		\label{sec:conclusion}
		
		In this work, we have shed light on an interesting first-order linear symmetric hyperbolic model with a rich mathematical structure: different divergence/curl involutions, several conserved and dissipated quantities, asymptotic consistency with the parabolic heat equation, etc. This set of properties makes the model a worthwhile benchmark for structure-preserving numerical methods, whose development and refinement remain an active field of research.          
		The semi-implicit numerical method we proposed here, based on compatible discrete operators on staggered grids, provably preserves all the mathematical properties of the model. Asymptotic preservation is also rigorously demonstrated, and it yields convergence orders that depend on how the initial data scales with the stiff characteristic speed, which was also validated numerically.  
		Several extensions and improvements suggest themselves. First, an extension to nonlinear equations and to continuum mechanics would make the model more meaningful at the price of additional complexity. We have in mind, in particular, the general models of the SHTC framework \cite{romenski2007conservative,peshkov2018continuum,peshkov2019continuum}, hyperbolic Euler-Fourier equations \cite{dhaouadi2024eulerian}, hyperbolic Cahn-Hilliard equations \cite{dhaouadi2025first}, etc. A recent contribution has extended the source-free version of the approach presented here to the nonlinear case \cite{lucca2026structure}. A combination of both works may lead to new results.    
		At the numerical level, extending the proven results to more general and complex meshes is of utmost importance to handle more general geometries. Higher-order extensions are also sought for the sake of efficiency. Both may be achieved with the aid of the approach recently introduced in \cite{abgrall2025simple}.

		\section*{Acknowledgments}
		This research was funded by the Italian Ministry of Education, University and Research (MIUR) in the frame of the Departments of Excellence Initiative 2018--2027 attributed to DICAM of the University of Trento (grant L. 232/2016). 
		FD was also funded by NextGenerationEU, Azione 247 MUR Young Researchers -- SoE line. 
		LRM was funded by the European Union's Horizon Europe research and innovation programme under the Marie Sk\l{}odowska-Curie fellowship COPERNICUS, grant agreement No. 101207132. MD is a member of the Gruppo Nazionale Calcolo Scientifico-Istituto Nazionale di Alta Matematica (GNCS-INdAM). 
		MD was funded by the Fondazione Caritro via the project SOPHOS and by the European Research Council (ERC) under the European Union's Horizon Europe research and innovation programme via the SOPHOS project, grant agreement No. ERC-ADG-2025-101265878-SOPHOS.
		  
		Views and opinions expressed are however those of the author(s) only and do not necessarily reflect those of the European Union or the European Research Council. Neither the European Union nor the granting authority can be held responsible for them.
		
		\bibliographystyle{mystyle}  
		\bibliography{./references}

		\appendix
		\section{Exact solution for the Cattaneo system in one dimension}\label{app:exact_sol}
		Recall that the initial value problem for the heat equation in one space dimension, defined on $\RR\times\RR^+$ by 
		\begin{subequations}
			\begin{gather}
				\pdt{\tilde{T}} - \kappa \pd{^2\tilde T}{x^2} =0,\label{eq:Fourier_EQ} \\  
				\tilde T(x,0) = T_0 + T_1\sin\prn{{\omega x}}, \quad \omega\in \RR 
				\label{eq:Fourier_IC}
			\end{gather}
		\end{subequations}
		admits the solution 
		\begin{equation}
			\tilde T(x,t) = T_0 + T_1\sin\prn{\omega x}e^{-\tilde\alpha t},  \quad \text{with} \quad \tilde{\alpha} =  \omega^2\kappa.
			\label{eq:exact_sol_Fourier}
		\end{equation}
		We provide here an analogous exact solution for the heat subsystem (\ref{eq:T2}-\ref{eq:J2}), which in one dimension reduces to the classical Cattaneo--Vernotte system, owing to the vanishing of curl terms
		\begin{subequations}
			\begin{align}
				&\pdt{T} + c_0\,\pd{J}{x} =0, \label{eq:Cattaneo_1}\\
				&\pdt{J} +c_0\,T_x = -\frac{c_0^2}{\kappa}J \label{eq:Cattaneo_2}
			\end{align}
			\label{eq:IVP_hyp}
		\end{subequations}
		We are particularly interested in the case where the initial datum for the temperature is taken as in \eqref{eq:Fourier_IC} and we look for an exact solution for the temperature of the form  
		\begin{equation*}
			T(x,t) = T_0 + T_1\sin\prn{\omega x}e^{-\alpha t},
		\end{equation*}
		where $\alpha > 0$ is to be determined. Inserting this ansatz into equation \eqref{eq:Cattaneo_1} and integrating over space yields $J$ up to a time-dependent integration constant $g(t)$
		\begin{equation*}
			J(x,t) = -\frac{\alpha}{\omega c_0} T_1\cos\prn{\omega x}e^{-\alpha t} + g(t).
		\end{equation*}
		Substituting both expressions into \eqref{eq:Cattaneo_2} yields
		\begin{equation*}
			\prn{\alpha^2- \frac{c_0^2}{\kappa}\alpha + \omega^2 c_0^2 } T_1\cos\prn{\omega x}e^{-\alpha t}  =  -\omega c_0\prn{g'(t) + \frac{c_0^2}{\kappa} g(t)}.
		\end{equation*}
		Since the left-hand side is proportional to $\cos(\omega x)$, while the right-hand side is only time-dependent, this equality can only hold if both sides vanish identically, \textit{i.e.},  
		\begin{equation*}
			\alpha^2- \frac{c_0^2}{\kappa}\alpha + \omega^2 c_0^2 = 0, \quad \text{and} \quad g'(t) + \frac{c_0^2}{\kappa} g(t) = 0,
		\end{equation*} 
		which gives, after the unstable modes have been excluded,
		\begin{equation*}
			\alpha = \frac{2 \tilde \alpha}{1 + \sqrt{1-4\omega^2\kappa^2/c_0^{2}}}, \qquad g(t) = J_0\,e^{-tc_0^2/\kappa}, \quad J_0\in\RR
		\end{equation*}
		Therefore, one obtains the following solution 
		\begin{align*}
			& T(x,t) = T_1\sin\prn{\omega x}e^{-\alpha t} + T_0, \\
			& J(x,t) = J_1\cos\prn{\omega x}e^{-\alpha t} + J_0\,e^{-tc_0^2/\kappa}, 
		\end{align*}
		where $J_1 = -\alpha T_1/(\omega c_0)$. 
		Note that in the limit $c_0\to+\infty$, $\alpha\to\tilde \alpha$ and hence $T(x,t)\to\tilde T(x,t)$.

	\end{document}